\documentclass[a4paper,11pt]{amsart}
\usepackage{graphicx,amssymb,amscd}
\usepackage{amsmath,amscd}
\usepackage[english]{babel}
\usepackage[applemac]{inputenc}
\usepackage[notref,notcite]{showkeys}
\newtheorem{thm}{Theorem}[section]
\newtheorem{cor}[thm]{Corollary}
\newtheorem{lem}[thm]{Lemma}
\newtheorem{prop}[thm]{Proposition}
\newtheorem{prop and def}[thm]{Proposition and Definition}
\newtheorem*{thm A}{Theorem A}
\newtheorem*{thm B}{Theorem B}
\theoremstyle{definition}
\newtheorem{defn}[thm]{Definition}

\newtheorem{example}[thm]{Example}

\theoremstyle{remark}
\newtheorem{rem}[thm]{Remark}
\newtheorem{conj}[thm]{Conjecture}

\def\cal{\mathcal}
\def\bb{\mathbb} 

\def\a{\alpha } 
\def\b{\beta }
\def\g{\gamma }
\def\d{\delta }
\def\D{\Delta}
\def\e{\epsilon }
\def\g{\gamma }
\def\G{\Gamma }
\def\Z{\mathbb Z}

\def\l{\lambda }

\def\o{\omega }
\def\O{\Omega }
\def\r{\rho }
\def\s{\sigma }

\def\t{\theta }

\def\f{\varphi }

\def\iff{ if and only if }

\def\dps{\displaystyle }

\def\det{{\rm det}\ }

\def\omlog{\theta}

\def\ot{\otimes }
\def\part{\partial }

\newcommand{\tlowername}[2]%
{$\stackrel{\makebox[1pt]{#1}}%
{\begin{picture}(0,0)%
\put(0,0){\makebox(0,6)[t]{\makebox[1pt]{$#2$}}}%
\end{picture}}$}%

\newcommand{\AR}[1]%
{\begin{picture}(#1,0)%
\put(0,0){\vector(1,0){#1}}%
\end{picture}}%

\newcommand{\DOTAR}[1]%
{\NUMBEROFDOTS=#1%
\divide\NUMBEROFDOTS by 3%
\begin{picture}(#1,0)%
\multiput(0,0)(3,0){\NUMBEROFDOTS}{\circle*{1}}%
\put(#1,0){\vector(1,0){0}}%
\end{picture}}%

\newcommand{\MONO}[1]%
{\begin{picture}(#1,0)%
\put(0,0){\vector(1,0){#1}}%
\put(2,-2){\line(0,1){4}}%
\end{picture}}%

\newcommand{\EPI}[1]%
{\begin{picture}(#1,0)(-#1,0)%
\put(-#1,0){\vector(1,0){#1}}%
\put(-6,-2){\line(0,1){4}}%
\end{picture}}%

\newcommand{\BIMO}[1]%
{\begin{picture}(#1,0)(-#1,0)%
\put(-#1,0){\vector(1,0){#1}}%
\put(-6,-2){\line(0,1){4}}%
\put(-#1,-2){\hspace{2pt}\line(0,1){4}}%
\end{picture}}%

\newcommand{\BIAR}[1]%
{\begin{picture}(#1,4)%
\put(0,0){\vector(1,0){#1}}%
\put(0,4){\vector(1,0){#1}}%
\end{picture}}%

\newcommand{\EQL}[1]%
{\begin{picture}(#1,0)%
\put(0,1){\line(1,0){#1}}%
\put(0,-1){\line(1,0){#1}}%
\end{picture}}%

\newcommand{\ADJAR}[1]%
{\begin{picture}(#1,4)%
\put(0,0){\vector(1,0){#1}}%
\put(#1,4){\vector(-1,0){#1}}%
\end{picture}}%

\newcommand{\BKAR}[1]%
{\begin{picture}(#1,0)%
\put(#1,0){\vector(-1,0){#1}}%
\end{picture}}%

\newcommand{\BKDOTAR}[1]%
{\NUMBEROFDOTS=#1%
\divide\NUMBEROFDOTS by 3%
\begin{picture}(#1,0)%
\multiput(#1,0)(-3,0){\NUMBEROFDOTS}{\circle*{1}}%
\put(0,0){\vector(-1,0){0}}%
\end{picture}}%

\newcommand{\BKMONO}[1]%
{\begin{picture}(#1,0)(-#1,0)%
\put(0,0){\vector(-1,0){#1}}%
\put(-2,-2){\line(0,1){4}}%
\end{picture}}%

\newcommand{\BKEPI}[1]%
{\begin{picture}(#1,0)%
\put(#1,0){\vector(-1,0){#1}}%
\put(6,-2){\line(0,1){4}}%
\end{picture}}%

\newcommand{\BKBIMO}[1]%
{\begin{picture}(#1,0)%
\put(#1,0){\vector(-1,0){#1}}%
\put(6,-2){\line(0,1){4}}%
\put(#1,-2){\hspace{-2pt}\line(0,1){4}}%
\end{picture}}%

\newcommand{\BKBIAR}[1]%
{\begin{picture}(#1,4)%
\put(#1,0){\vector(-1,0){#1}}%
\put(#1,4){\vector(-1,0){#1}}%
\end{picture}}%

\newcommand{\BKADJAR}[1]%
{\begin{picture}(#1,4)%
\put(0,4){\vector(1,0){#1}}%
\put(#1,0){\vector(-1,0){#1}}%
\end{picture}}%

\newcommand{\lowername}[2]%
{$\stackrel{\makebox[1pt]{#1}}%
{\begin{picture}(0,0)%
\truex{600}%
\put(0,0){\makebox(0,\value{x})[t]{\makebox[1pt]{$#2$}}}%
\end{picture}}$}%

\newcommand{\hcase}[2]%
{\makebox[0pt]%
{\raisebox{-1pt}[0pt][0pt]{#1{#2}}}}%

\newcommand{\Hcase}[3]%
{\makebox[0pt]
{\raisebox{-1pt}[0pt][0pt]%
{$\stackrel{\makebox[0pt]{$\textstyle{#2}$}}{#1{#3}}$}}}%

\newcommand{\hcasE}[3]%
{\makebox[0pt]%
{\raisebox{-9pt}[0pt][0pt]%
{\lowername{#1{#3}}{#2}}}}%

\newcommand{\hbicase}[2]%
{\makebox[0pt]%
{\raisebox{-2.5pt}[0pt][0pt]{#1{#2}}}}%

\newcommand{\Hbicase}[4]%
{\makebox[0pt]
{\raisebox{-10.5pt}[0pt][0pt]%
{$\stackrel{\makebox[0pt]{$\textstyle{#2}$}}%
{\mbox{\lowername{#1{#4}}{#3}}}$}}}%

\newcommand{\EAR}[1]%
{\begin{picture}(#1,0)%
\put(0,0){\vector(1,0){#1}}%
\end{picture}}%

\newcommand{\EDOTAR}[1]%
{\truex{100}\truey{300}%
\NUMBEROFDOTS=#1%
\divide\NUMBEROFDOTS by \value{y}%
\begin{picture}(#1,0)%
\multiput(0,0)(\value{y},0){\NUMBEROFDOTS}%
{\circle*{\value{x}}}%
\put(#1,0){\vector(1,0){0}}%
\end{picture}}%

\newcommand{\EMONO}[1]%
{\begin{picture}(#1,0)%
\put(0,0){\vector(1,0){#1}}%
\truex{300}\truey{600}%
\put(\value{x},-\value{x}){\line(0,1){\value{y}}}%
\end{picture}}%

\newcommand{\EEPI}[1]%
{\begin{picture}(#1,0)(-#1,0)%
\put(-#1,0){\vector(1,0){#1}}%
\truex{300}\truey{600}\truez{800}%
\put(-\value{z},-\value{x}){\line(0,1){\value{y}}}%
\end{picture}}%

\newcommand{\EBIMO}[1]%
{\begin{picture}(#1,0)(-#1,0)%
\put(-#1,0){\vector(1,0){#1}}%
\truex{300}\truey{600}\truez{800}%
\put(-\value{z},-\value{x}){\line(0,1){\value{y}}}%
\put(-#1,-\value{x}){\hspace{3pt}\line(0,1){\value{y}}}%
\end{picture}}%

\newcommand{\EBIAR}[1]%
{\truex{400}%
\begin{picture}(#1,\value{x})%
\put(0,0){\vector(1,0){#1}}%
\put(0,\value{x}){\vector(1,0){#1}}%
\end{picture}}%

\newcommand{\EEQL}[1]%
{\begin{picture}(#1,0)%
\truex{200}%
\put(0,\value{x}){\line(1,0){#1}}%
\put(0,0){\line(1,0){#1}}%
\end{picture}}%

\newcommand{\EADJAR}[1]%
{\truex{400}%
\begin{picture}(#1,\value{x})%
\put(0,0){\vector(1,0){#1}}%
\put(#1,\value{x}){\vector(-1,0){#1}}%
\end{picture}}%

\newcommand{\ear}%
{\hspace{\SOURCE\unitlength}%
\hcase{\EAR}{\ARROWLENGTH}}%

\newcommand{\Earv}[2]{\Hcase{\EAR}{#1}{#200}}%

\newcommand{\Ear}[1]%
{\hspace{\SOURCE\unitlength}%
\Hcase{\EAR}{#1}{\ARROWLENGTH}}%

\newcommand{\eaR}[1]%
{\hspace{\SOURCE\unitlength}%
\hcasE{\EAR}{#1}{\ARROWLENGTH}}%

\newcommand{\edotar}%
{\hspace{\SOURCE\unitlength}%
\hcase{\EDOTAR}{\ARROWLENGTH}}%

\newcommand{\Edotar}[1]%
{\hspace{\SOURCE\unitlength}%
\Hcase{\EDOTAR}{#1}{\ARROWLENGTH}}%

\newcommand{\edotaR}[1]%
{\hspace{\SOURCE\unitlength}%
\hcasE{\EDOTAR}{#1}{\ARROWLENGTH}}%

\newcommand{\emono}%
{\hspace{\SOURCE\unitlength}%
\hcase{\EMONO}{\ARROWLENGTH}}%

\newcommand{\Emono}[1]%
{\hspace{\SOURCE\unitlength}%
\Hcase{\EMONO}{#1}{\ARROWLENGTH}}%

\newcommand{\emonO}[1]%
{\hspace{\SOURCE\unitlength}%
\hcasE{\EMONO}{#1}{\ARROWLENGTH}}%

\newcommand{\eepi}%
{\hspace{\SOURCE\unitlength}%
\hcase{\EEPI}{\ARROWLENGTH}}%

\newcommand{\Eepi}[1]%
{\hspace{\SOURCE\unitlength}%
\Hcase{\EEPI}{#1}{\ARROWLENGTH}}%

\newcommand{\eepI}[1]%
{\hspace{\SOURCE\unitlength}%
\hcasE{\EEPI}{#1}{\ARROWLENGTH}}%

\newcommand{\ebimo}%
{\hspace{\SOURCE\unitlength}%
\hcase{\EBIMO}{\ARROWLENGTH}}%

\newcommand{\Ebimo}[1]%
{\hspace{\SOURCE\unitlength}%
\Hcase{\EBIMO}{#1}{\ARROWLENGTH}}%

\newcommand{\ebimO}[1]%
{\hspace{\SOURCE\unitlength}%
\hcasE{\EBIMO}{#1}{\ARROWLENGTH}}%

\newcommand{\eiso}%
{\hspace{\SOURCE\unitlength}%
\Hcase{\EAR}{\cong}{\ARROWLENGTH}}%

\newcommand{\Eiso}[1]%
{\hspace{\SOURCE\unitlength}%
\Hcase{\EAR}{\cong#1}{\ARROWLENGTH}}%

\newcommand{\eisO}[1]%
{\hspace{\SOURCE\unitlength}%
\hcasE{\EAR}{\cong#1}{\ARROWLENGTH}}%

\newcommand{\ebiar}%
{\hspace{\SOURCE\unitlength}%
\hbicase{\EBIAR}{\ARROWLENGTH}}%

\newcommand{\Ebiar}[2]%
{\hspace{\SOURCE\unitlength}%
\Hbicase{\EBIAR}{#1}{#2}{\ARROWLENGTH}}%

\newcommand{\eeql}%
{\hspace{\SOURCE\unitlength}%
\hbicase{\EEQL}{\ARROWLENGTH}}%

\newcommand{\eadjar}%
{\hspace{\SOURCE\unitlength}%
\hbicase{\EADJAR}{\ARROWLENGTH}}%

\newcommand{\Eadjar}[2]%
{\hspace{\SOURCE\unitlength}%
\Hbicase{\EADJAR}{#1}{#2}{\ARROWLENGTH}}%

\newcommand{\WAR}[1]%
{\begin{picture}(#1,0)%
\put(#1,0){\vector(-1,0){#1}}%
\end{picture}}%

\newcommand{\WDOTAR}[1]%
{\truex{100}\truey{300}%
\NUMBEROFDOTS=#1%
\divide\NUMBEROFDOTS by \value{y}%
\begin{picture}(#1,0)%
\multiput(#1,0)(-\value{y},0){\NUMBEROFDOTS}%
{\circle*{\value{x}}}%
\put(0,0){\vector(-1,0){0}}%
\end{picture}}%

\newcommand{\WMONO}[1]%
{\begin{picture}(#1,0)(-#1,0)%
\put(0,0){\vector(-1,0){#1}}%
\truex{300}\truey{600}%
\put(-\value{x},-\value{x}){\line(0,1){\value{y}}}%
\end{picture}}%

\newcommand{\WEPI}[1]%
{\begin{picture}(#1,0)%
\put(#1,0){\vector(-1,0){#1}}%
\truex{300}\truey{600}\truez{800}%
\put(\value{z},-\value{x}){\line(0,1){\value{y}}}%
\end{picture}}%

\newcommand{\WBIMO}[1]%
{\begin{picture}(#1,0)%
\put(#1,0){\vector(-1,0){#1}}%
\truex{300}\truey{600}\truez{800}%
\put(\value{z},-\value{x}){\line(0,1){\value{y}}}%
\put(#1,-\value{x}){\hspace{-3pt}\line(0,1){\value{y}}}%
\end{picture}}%

\newcommand{\WBIAR}[1]%
{\truex{400}%
\begin{picture}(#1,\value{x})%
\put(#1,0){\vector(-1,0){#1}}%
\put(#1,\value{x}){\vector(-1,0){#1}}%
\end{picture}}%

\newcommand{\WADJAR}[1]%
{\truex{400}%
\begin{picture}(#1,\value{x})%
\put(0,\value{x}){\vector(1,0){#1}}%
\put(#1,0){\vector(-1,0){#1}}%
\end{picture}}%

\newcommand{\war}%
{\hspace{\SOURCE\unitlength}%
\hcase{\WAR}{\ARROWLENGTH}}%

\newcommand{\War}[1]%
{\hspace{\SOURCE\unitlength}%
\Hcase{\WAR}{#1}{\ARROWLENGTH}}%

\newcommand{\waR}[1]%
{\hspace{\SOURCE\unitlength}%
\hcasE{\WAR}{#1}{\ARROWLENGTH}}%

\newcommand{\wdotar}%
{\hspace{\SOURCE\unitlength}%
\hcase{\WDOTAR}{\ARROWLENGTH}}%

\newcommand{\Wdotar}[1]%
{\hspace{\SOURCE\unitlength}%
\Hcase{\WDOTAR}{#1}{\ARROWLENGTH}}%

\newcommand{\wdotaR}[1]%
{\hspace{\SOURCE\unitlength}%
\hcasE{\WDOTAR}{#1}{\ARROWLENGTH}}%

\newcommand{\wmono}%
{\hspace{\SOURCE\unitlength}%
\hcase{\WMONO}{\ARROWLENGTH}}%

\newcommand{\Wmono}[1]%
{\hspace{\SOURCE\unitlength}%
\Hcase{\WMONO}{#1}{\ARROWLENGTH}}%

\newcommand{\wmonO}[1]%
{\hspace{\SOURCE\unitlength}%
\hcasE{\WMONO}{#1}{\ARROWLENGTH}}%

\newcommand{\wepi}%
{\hspace{\SOURCE\unitlength}%
\hcase{\WEPI}{\ARROWLENGTH}}%

\newcommand{\Wepi}[1]%
{\hspace{\SOURCE\unitlength}%
\Hcase{\WEPI}{#1}{\ARROWLENGTH}}%

\newcommand{\wepI}[1]%
{\hspace{\SOURCE\unitlength}%
\hcasE{\WEPI}{#1}{\ARROWLENGTH}}%

\newcommand{\wbimo}%
{\hspace{\SOURCE\unitlength}%
\hcase{\WBIMO}{\ARROWLENGTH}}%

\newcommand{\Wbimo}[1]%
{\hspace{\SOURCE\unitlength}%
\Hcase{\WBIMO}{#1}{\ARROWLENGTH}}%

\newcommand{\wbimO}[1]%
{\hspace{\SOURCE\unitlength}%
\hcasE{\WBIMO}{#1}{\ARROWLENGTH}}%

\newcommand{\wiso}%
{\hspace{\SOURCE\unitlength}%
\Hcase{\WAR}{\cong}{\ARROWLENGTH}}%

\newcommand{\Wiso}[1]%
{\hspace{\SOURCE\unitlength}%
\Hcase{\WAR}{#1}{\ARROWLENGTH}}%

\newcommand{\wisO}[1]%
{\hspace{\SOURCE\unitlength}%
\hcasE{\WAR}{#1}{\ARROWLENGTH}}%

\newcommand{\wbiar}%
{\hspace{\SOURCE\unitlength}%
\hbicase{\WBIAR}{\ARROWLENGTH}}%

\newcommand{\Wbiar}[2]%
{\hspace{\SOURCE\unitlength}%
\Hbicase{\WBIAR}{#1}{#2}{\ARROWLENGTH}}%

\newcommand{\weql}%
{\hspace{\SOURCE\unitlength}%
\hbicase{\EEQL}{\ARROWLENGTH}}%

\newcommand{\wadjar}%
{\hspace{\SOURCE\unitlength}%
\hbicase{\WADJAR}{\ARROWLENGTH}}%

\newcommand{\Wadjar}[2]%
{\hspace{\SOURCE\unitlength}%
\Hbicase{\WADJAR}{#1}{#2}{\ARROWLENGTH}}%

\newcommand{\Vcase}[3]{\makebox[0pt]%
{\makebox[0pt][r]{\raisebox{0pt}[0pt][0pt]{${#2}\hspace{2pt}$}}}#1{#3}}%

\newcommand{\vcasE}[3]{\makebox[0pt]%
{#1{#3}\makebox[0pt][l]{\raisebox{0pt}[0pt][0pt]{\hspace{2pt}$#2$}}}}%

\newcommand{\Vbicase}[4]{\makebox[0pt]%
{\makebox[0pt][r]{\raisebox{0pt}[0pt][0pt]{$#2$\hspace{4pt}}}#1{#4}%
\makebox[0pt][l]{\raisebox{0pt}[0pt][0pt]{\hspace{5pt}$#3$}}}}%

\newcommand{\SAR}[1]%
{\begin{picture}(0,0)%
\put(0,0){\makebox(0,0)%
{\begin{picture}(0,#1)%
\put(0,#1){\vector(0,-1){#1}}%
\end{picture}}}\end{picture}}%

\newcommand{\SDOTAR}[1]%
{\truex{100}\truey{300}%
\NUMBEROFDOTS=#1%
\divide\NUMBEROFDOTS by \value{y}%
\begin{picture}(0,0)%
\put(0,0){\makebox(0,0)%
{\begin{picture}(0,#1)%
\multiput(0,#1)(0,-\value{y}){\NUMBEROFDOTS}%
{\circle*{\value{x}}}%
\put(0,0){\vector(0,-1){0}}%
\end{picture}}}\end{picture}}%

\newcommand{\SMONO}[1]%
{\begin{picture}(0,0)%
\put(0,0){\makebox(0,0)%
{\begin{picture}(0,#1)%
\put(0,#1){\vector(0,-1){#1}}%
\truex{300}\truey{600}%
\put(0,#1){\begin{picture}(0,0)%
\put(-\value{x},-\value{x}){\line(1,0){\value{y}}}\end{picture}}%
\end{picture}}}\end{picture}}%

\newcommand{\SEPI}[1]%
{\begin{picture}(0,0)%
\put(0,0){\makebox(0,0)%
{\begin{picture}(0,#1)%
\put(0,#1){\vector(0,-1){#1}}%
\truex{300}\truey{600}\truez{800}%
\put(-\value{x},\value{z}){\line(1,0){\value{y}}}%
\end{picture}}}\end{picture}}%

\newcommand{\SBIMO}[1]%
{\begin{picture}(0,0)%
\put(0,0){\makebox(0,0)%
{\begin{picture}(0,#1)%
\put(0,#1){\vector(0,-1){#1}}%
\truex{300}\truey{600}\truez{800}%
\put(0,#1){\begin{picture}(0,0)%
\put(-\value{x},-\value{x}){\line(1,0){\value{y}}}\end{picture}}%
\put(-\value{x},\value{z}){\line(1,0){\value{y}}}%
\end{picture}}}\end{picture}}%

\newcommand{\SBIAR}[1]%
{\begin{picture}(0,0)%
\truex{200}%
\put(0,0){\makebox(0,0)%
{\begin{picture}(0,#1)\put(-\value{x},#1){\vector(0,-1){#1}}%
\put(\value{x},#1){\vector(0,-1){#1}}%
\end{picture}}}\end{picture}}%

\newcommand{\SEQL}[1]%
{\begin{picture}(0,0)%
\truex{100}%
\put(0,0){\makebox(0,0)%
{\begin{picture}(0,#1)\put(-\value{x},#1){\line(0,-1){#1}}%
\put(\value{x},#1){\line(0,-1){#1}}%
\end{picture}}}\end{picture}}%

\newcommand{\Sarv}[2]{\Vcase{\SAR}{#1}{#200}}%

\newcommand{\Sar}[1]{\Sarv{#1}{50}}%

\newcommand{\Sisov}[2]%
{\Vbicase{\SAR}{#1\hspace{-2pt}}{\hspace{-2pt}\cong}{#200}}%

\newcommand{\NAR}[1]%
{\begin{picture}(0,0)%
\put(0,0){\makebox(0,0)%
{\begin{picture}(0,#1)\put(0,0){\vector(0,1){#1}}%
\end{picture}}}\end{picture}}%

\newcommand{\NDOTAR}[1]%
{\truex{100}\truey{300}%
\NUMBEROFDOTS=#1%
\divide\NUMBEROFDOTS by \value{y}%
\begin{picture}(0,0)%
\put(0,0){\makebox(0,0)%
{\begin{picture}(0,#1)%
\multiput(0,0)(0,\value{y}){\NUMBEROFDOTS}%
{\circle*{\value{x}}}%
\put(0,#1){\vector(0,1){0}}%
\end{picture}}}\end{picture}}%

\newcommand{\NMONO}[1]%
{\begin{picture}(0,0)%
\put(0,0){\makebox(0,0)%
{\begin{picture}(0,#1)%
\put(0,0){\vector(0,1){#1}}%
\truex{300}\truey{600}%
\put(-\value{x},\value{x}){\line(1,0){\value{y}}}%
\end{picture}}}%
\end{picture}}%

\newcommand{\NEPI}[1]%
{\begin{picture}(0,0)%
\put(0,0){\makebox(0,0)%
{\begin{picture}(0,#1)%
\put(0,0){\vector(0,1){#1}}%
\truex{300}\truey{600}\truez{800}%
\put(0,#1){\begin{picture}(0,0)%
\put(-\value{x},-\value{z}){\line(1,0){\value{y}}}\end{picture}}%
\end{picture}}}\end{picture}}%

\newcommand{\NBIMO}[1]%
{\begin{picture}(0,0)%
\put(0,0){\makebox(0,0)%
{\begin{picture}(0,#1)%
\put(0,0){\vector(0,1){#1}}%
\truex{300}\truey{600}\truez{800}%
\put(-\value{x},\value{x}){\line(1,0){\value{y}}}%
\put(0,#1){\begin{picture}(0,0)%
\put(-\value{x},-\value{z}){\line(1,0){\value{y}}}\end{picture}}%
\end{picture}}}\end{picture}}%

\newcommand{\NBIAR}[1]%
{\begin{picture}(0,0)%
\truex{200}%
\put(0,0){\makebox(0,0)%
{\begin{picture}(0,#1)\put(-\value{x},0){\vector(0,1){#1}}%
\put(\value{x},0){\vector(0,1){#1}}%
\end{picture}}}\end{picture}}%

\newcommand{\naRv}[2]{\vcasE{\NAR}{#1}{#200}}%

\newcommand{\naR}[1]{\naRv{#1}{50}}%

\newcommand{\Nisov}[2]%
{\Vbicase{\NAR}{#1\hspace{-2pt}}{\hspace{-2pt}\cong}{#200}}%

\newcommand{\fdcase}[3]{\begin{picture}(0,0)%
\put(0,-150){#1}%
\truex{200}\truey{600}\truez{600}%
\put(-\value{x},-\value{x}){\makebox(0,\value{z})[r]{${#2}$}}%
\put(\value{x},-\value{y}){\makebox(0,\value{z})[l]{${#3}$}}%
\end{picture}}%

\newcommand{\NEAR}{\begin{picture}(0,0)%
\put(-2900,-2900){\vector(1,1){5800}}%
\end{picture}}%

\newcommand{\NEDOTAR}%
{\truex{100}\truey{212}%
\NUMBEROFDOTS=5800%
\divide\NUMBEROFDOTS by \value{y}%
\begin{picture}(0,0)%
\multiput(-2900,-2900)(\value{y},\value{y}){\NUMBEROFDOTS}%
{\circle*{\value{x}}}%
\put(2900,2900){\vector(1,1){0}}%
\end{picture}}%

\newcommand{\Near}[1]{\fdcase{\NEAR}{#1}{}}%

\newcommand{\SWAR}{\begin{picture}(0,0)%
\put(2900,2900){\vector(-1,-1){5800}}%
\end{picture}}%

\newcommand{\SWDOTAR}%
{\truex{100}\truey{212}%
\NUMBEROFDOTS=5800%
\divide\NUMBEROFDOTS by \value{y}%
\begin{picture}(0,0)%
\multiput(2900,2900)(-\value{y},-\value{y}){\NUMBEROFDOTS}%
{\circle*{\value{x}}}%
\put(-2900,-2900){\vector(-1,-1){0}}%
\end{picture}}%

\newcommand{\swaR}[1]{\fdcase{\SWAR}{}{#1}}%

\newcommand{\sdcase}[3]{\begin{picture}(0,0)%
\put(0,-150){#1}%
\truex{100}\truez{600}%
\put(\value{x},\value{x}){\makebox(0,\value{z})[l]{${#2}$}}%
\truex{300}\truey{800}%
\put(-\value{x},-\value{y}){\makebox(0,\value{z})[r]{${#3}$}}%
\end{picture}}%

\newcommand{\SEAR}{\begin{picture}(0,0)%
\put(-2900,2900){\vector(1,-1){5800}}%
\end{picture}}%

\newcommand{\SEDOTAR}%
{\truex{100}\truey{212}%
\NUMBEROFDOTS=5800%
\divide\NUMBEROFDOTS by \value{y}%
\begin{picture}(0,0)%
\multiput(-2900,2900)(\value{y},-\value{y}){\NUMBEROFDOTS}%
{\circle*{\value{x}}}%
\put(2900,-2900){\vector(1,-1){0}}%
\end{picture}}%

\newcommand{\Sear}[1]{\sdcase{\SEAR}{#1}{}}%

\newcommand{\seaR}[1]{\sdcase{\SEAR}{}{#1}}%

\newcommand{\NWDOTAR}%
{\truex{100}\truey{212}%
\NUMBEROFDOTS=5800%
\divide\NUMBEROFDOTS by \value{y}%
\begin{picture}(0,0)%
\multiput(2900,-2900)(-\value{y},\value{y}){\NUMBEROFDOTS}%
{\circle*{\value{x}}}%
\put(-2900,2900){\vector(-1,1){0}}%
\end{picture}}%

\newcommand{\ENEAR}[2]%
{\makebox[0pt]{\begin{picture}(0,0)%
\put(0,-150){\makebox(0,0){\begin{picture}(0,0)%
\put(-6600,-3300){\vector(2,1){13200}}%
\truex{200}\truey{800}\truez{600}%
\put(-\value{x},\value{x}){\makebox(0,\value{z})[r]{${#1}$}}%
\put(\value{x},-\value{y}){\makebox(0,\value{z})[l]{${#2}$}}%
\end{picture}}}\end{picture}}}%

\newcommand{\ESEAR}[2]%
{\makebox[0pt]{\begin{picture}(0,0)%
\put(0,-150){\makebox(0,0){\begin{picture}(0,0)%
\put(-6600,3300){\vector(2,-1){13200}}%
\truex{200}\truey{800}\truez{600}%
\put(\value{x},\value{x}){\makebox(0,\value{z})[l]{${#1}$}}%
\put(-\value{x},-\value{y}){\makebox(0,\value{z})[r]{${#2}$}}%
\end{picture}}}\end{picture}}}%

\newcommand{\WNWAR}[2]%
{\makebox[0pt]{\begin{picture}(0,0)%
\put(0,-150){\makebox(0,0){\begin{picture}(0,0)%
\put(6600,-3300){\vector(-2,1){13200}}%
\truex{200}\truey{800}\truez{600}%
\put(\value{x},\value{x}){\makebox(0,\value{z})[l]{${#1}$}}%
\put(-\value{x},-\value{y}){\makebox(0,\value{z})[r]{${#2}$}}%
\end{picture}}}\end{picture}}}%

\newcommand{\WSWAR}[2]%
{\makebox[0pt]{\begin{picture}(0,0)%
\put(0,-150){\makebox(0,0){\begin{picture}(0,0)%
\put(6600,3300){\vector(-2,-1){13200}}%
\truex{200}\truey{800}\truez{600}%
\put(-\value{x},\value{x}){\makebox(0,\value{z})[r]{${#1}$}}%
\put(\value{x},-\value{y}){\makebox(0,\value{z})[l]{${#2}$}}%
\end{picture}}}\end{picture}}}%

\newcommand{\NNEAR}[2]%
{\raisebox{-1pt}[0pt][0pt]{\begin{picture}(0,0)%
\put(0,0){\makebox(0,0){\begin{picture}(0,0)%
\put(-3300,-6600){\vector(1,2){6600}}%
\truex{100}\truez{600}%
\put(-\value{x},\value{x}){\makebox(0,\value{z})[r]{${#1}$}}%
\put(\value{x},-\value{z}){\makebox(0,\value{z})[l]{${#2}$}}%
\end{picture}}}\end{picture}}}%

\newcommand{\SSWAR}[2]%
{\raisebox{-1pt}[0pt][0pt]{\begin{picture}(0,0)%
\put(0,0){\makebox(0,0){\begin{picture}(0,0)%
\put(3300,6600){\vector(-1,-2){6600}}%
\truex{100}\truez{600}%
\put(-\value{x},\value{x}){\makebox(0,\value{z})[r]{${#1}$}}%
\put(\value{x},-\value{z}){\makebox(0,\value{z})[l]{${#2}$}}%
\end{picture}}}\end{picture}}}%

\newcommand{\SSEAR}[2]%
{\raisebox{-1pt}[0pt][0pt]{\begin{picture}(0,0)%
\put(0,0){\makebox(0,0){\begin{picture}(0,0)%
\put(-3300,6600){\vector(1,-2){6600}}%
\truex{200}\truez{600}%
\put(\value{x},\value{x}){\makebox(0,\value{z})[l]{${#1}$}}%
\put(-\value{x},-\value{z}){\makebox(0,\value{z})[r]{${#2}$}}%
\end{picture}}}\end{picture}}}%

\newcommand{\NNWAR}[2]%
{\raisebox{-1pt}[0pt][0pt]{\begin{picture}(0,0)%
\put(0,0){\makebox(0,0){\begin{picture}(0,0)%
\put(3300,-6600){\vector(-1,2){6600}}%
\truex{200}\truez{600}%
\put(\value{x},\value{x}){\makebox(0,\value{z})[l]{${#1}$}}%
\put(-\value{x},-\value{z}){\makebox(0,\value{z})[r]{${#2}$}}%
\end{picture}}}\end{picture}}}%

\newcommand{\Necurve}[2]%
{\begin{picture}(0,0)%
\truex{1300}\truey{2000}\truez{200}%
\put(0,\value{x}){\oval(#200,\value{y})[t]}%
\put(0,\value{x}){\makebox(0,0){\begin{picture}(#200,0)%
\put(#200,0){\vector(0,-1){\value{z}}}%
\put(0,0){\line(0,-1){\value{z}}}\end{picture}}}%
\truex{2500}%
\put(0,\value{x}){\makebox(0,0)[b]{${#1}$}}%
\end{picture}}%

\newcommand{\Nwcurve}[2]%
{\begin{picture}(0,0)%
\truex{1300}\truey{2000}\truez{200}%
\put(0,\value{x}){\oval(#200,\value{y})[t]}%
\put(0,\value{x}){\makebox(0,0){\begin{picture}(#200,0)%
\put(#200,0){\line(0,-1){\value{z}}}%
\put(0,0){\vector(0,-1){\value{z}}}\end{picture}}}%
\truex{2500}%
\put(0,\value{x}){\makebox(0,0)[b]{${#1}$}}%
\end{picture}}%

\newcommand{\Securve}[2]%
{\begin{picture}(0,0)%
\truex{1300}\truey{2000}\truez{200}%
\put(0,-\value{x}){\oval(#200,\value{y})[b]}%
\put(0,-\value{x}){\makebox(0,0){\begin{picture}(#200,0)%
\put(#200,0){\vector(0,1){\value{z}}}%
\put(0,0){\line(0,1){\value{z}}}\end{picture}}}%
\truex{2500}%
\put(0,-\value{x}){\makebox(0,0)[t]{${#1}$}}%
\end{picture}}%

\newcommand{\Swcurve}[2]%
{\begin{picture}(0,0)%
\truex{1300}\truey{2000}\truez{200}%
\put(0,-\value{x}){\oval(#200,\value{y})[b]}%
\put(0,-\value{x}){\makebox(0,0){\begin{picture}(#200,0)%
\put(#200,0){\line(0,1){\value{z}}}%
\put(0,0){\vector(0,1){\value{z}}}\end{picture}}}%
\truex{2500}%
\put(0,-\value{x}){\makebox(0,0)[t]{${#1}$}}%
\end{picture}}%

\newcommand{\Escurve}[2]%
{\begin{picture}(0,0)%
\truex{1400}\truey{2000}\truez{200}%
\put(\value{x},0){\oval(\value{y},#200)[r]}%
\put(\value{x},0){\makebox(0,0){\begin{picture}(0,#200)%
\put(0,0){\vector(-1,0){\value{z}}}%
\put(0,#200){\line(-1,0){\value{z}}}\end{picture}}}%
\truex{2500}%
\put(\value{x},0){\makebox(0,0)[l]{${#1}$}}%
\end{picture}}%

\newcommand{\Encurve}[2]%
{\begin{picture}(0,0)%
\truex{1400}\truey{2000}\truez{200}%
\put(\value{x},0){\oval(\value{y},#200)[r]}%
\put(\value{x},0){\makebox(0,0){\begin{picture}(0,#200)%
\put(0,0){\line(-1,0){\value{z}}}%
\put(0,#200){\vector(-1,0){\value{z}}}\end{picture}}}%
\truex{2500}%
\put(\value{x},0){\makebox(0,0)[l]{${#1}$}}%
\end{picture}}%

\newcommand{\Wscurve}[2]%
{\begin{picture}(0,0)%
\truex{1300}\truey{2000}\truez{200}%
\put(-\value{x},0){\oval(\value{y},#200)[l]}%
\put(-\value{x},0){\makebox(0,0){\begin{picture}(0,#200)%
\put(0,0){\vector(1,0){\value{z}}}%
\put(0,#200){\line(1,0){\value{z}}}\end{picture}}}%
\truex{2400}%
\put(-\value{x},0){\makebox(0,0)[r]{${#1}$}}%
\end{picture}}%

\newcommand{\Wncurve}[2]%
{\begin{picture}(0,0)%
\truex{1300}\truey{2000}\truez{200}%
\put(-\value{x},0){\oval(\value{y},#200)[l]}%
\put(-\value{x},0){\makebox(0,0){\begin{picture}(0,#200)%
\put(0,0){\line(1,0){\value{z}}}%
\put(0,#200){\vector(1,0){\value{z}}}\end{picture}}}%
\truex{2400}%
\put(-\value{x},0){\makebox(0,0)[r]{${#1}$}}%
\end{picture}}%

\newcount\SCALE%

\newcount\NUMBER%

\newcount\LINE%

\newcount\COLUMN%

\newcount\WIDTH%

\newcount\SOURCE%

\newcount\ARROW%

\newcount\TARGET%

\newcount\ARROWLENGTH%

\newcount\NUMBEROFDOTS%

\newcounter{x}%

\newcounter{y}%

\newcounter{z}%

\newcounter{horizontal}%

\newcounter{vertical}%

\newskip\itemlength%

\newskip\firstitem%

\newskip\seconditem%

\newcommand{\printarrow}{}%

\newcommand{\truex}[1]{%
\NUMBER=#1%
\multiply\NUMBER by 100%
\divide\NUMBER by \SCALE%
\setcounter{x}{\NUMBER}}%

\newcommand{\truey}[1]{%
\NUMBER=#1%
\multiply\NUMBER by 100%
\divide\NUMBER by \SCALE%
\setcounter{y}{\NUMBER}}%

\newcommand{\truez}[1]{%
\NUMBER=#1%
\multiply\NUMBER by 100%
\divide\NUMBER by \SCALE%
\setcounter{z}{\NUMBER}}%

\newcommand{\changecounters}[1]{%
\SOURCE=\ARROW%
\ARROW=\TARGET%
\settowidth{\itemlength}{#1}%
\ifdim \itemlength > 2800\unitlength%
\addtolength{\itemlength}{-2800\unitlength}%
\TARGET=\itemlength%
\divide\TARGET by 1310%
\multiply\TARGET by 100%
\divide\TARGET by \SCALE%
\else%
\TARGET=0%
\fi%
\ARROWLENGTH=5000%
\advance\ARROWLENGTH by -\SOURCE%
\advance\ARROWLENGTH by -\TARGET%
\advance\SOURCE by -\TARGET}%

\newcommand{\initialize}[1]{%
\LINE=0%
\COLUMN=0%
\WIDTH=0%
\ARROW=0%
\TARGET=0%
\changecounters{#1}%
\renewcommand{\printarrow}{#1}%
\begin{center}%
\vspace{10pt}%
\begin{picture}(0,0)}%

\newcommand{\DIAGV}[2]{%
\SCALE=#1%
\setlength{\unitlength}{655sp}%
\multiply\unitlength by \SCALE%
\divide\unitlength by 100%
\initialize{\mbox{$#2$}}}%

\newcommand{\n}[1]{%
\changecounters{\mbox{$#1$}}%
\put(\COLUMN,\LINE){\makebox(0,0){\printarrow}}%
\thinlines%
\renewcommand{\printarrow}{\mbox{$#1$}}%
\advance\COLUMN by 4000}%

\newcommand{\nn}[1]{%
\put(\COLUMN,\LINE){\makebox(0,0){\printarrow}}%
\thinlines%
\ifnum \WIDTH < \COLUMN%
\WIDTH=\COLUMN%
\else%
\fi%
\advance\LINE by -4000%
\COLUMN=0%
\ARROW=0%
\TARGET=0%
\changecounters{\mbox{$#1$}}%
\renewcommand{\printarrow}{\mbox{$#1$}}}%

\newcommand{\conclude}{%
\put(\COLUMN,\LINE){\makebox(0,0){\printarrow}}%
\thinlines%
\ifnum \WIDTH < \COLUMN%
\WIDTH=\COLUMN%
\else%
\fi%
\setcounter{horizontal}{\WIDTH}%
\setcounter{vertical}{-\LINE}%
\end{picture}}%

\newcommand{\diag}{%
\conclude%
\raisebox{0pt}[0pt][\value{vertical}\unitlength]{}%
\hspace*{\value{horizontal}\unitlength}%
\vspace{10pt}%
\end{center}%
\setlength{\unitlength}{1pt}}%

\newcommand{\diagv}[3]{%
\conclude%
\NUMBER=#1%
\rule{0pt}{\NUMBER pt}%
\hspace*{-#2pt}%
\raisebox{0pt}[0pt][\value{vertical}\unitlength]{}%
\hspace*{\value{horizontal}\unitlength}
\NUMBER=#3%
\advance\NUMBER by 10%
\vspace*{\NUMBER pt}%
\end{center}%
\setlength{\unitlength}{1pt}}%

\newcommand{\N}[1]%
{\raisebox{0pt}[7pt][0pt]{$#1$}}%

\newcommand{\crosslength}[2]{%
\settowidth{\firstitem}{#1}%
\settowidth{\seconditem}{#2}%
\ifdim\firstitem < \seconditem%
\itemlength=\seconditem%
\else%
\itemlength=\firstitem%
\fi%
\divide\itemlength by 2%
\hspace{\itemlength}}%

\begin{document}


\title{On Classification of compact complex surfaces of class VII}

\author[G.Dloussky]{Georges Dloussky}
\address{Georges Dloussky\\ Aix-Marseille Universit\'e,
Institut de Math\'ematiques de Marseille (I2M), UMR 7373,
Site de Saint Charles,
3 place Victor Hugo,
Case 19,
13331 Marseille cedex 3}
\email{georges.dloussky@univ-amu.fr}

\thanks{G.D. is grateful to CIRGET of UQAM in Montr\'eal and I2M, UMR  7373 of Aix-Marseille University for the financial support.   The author   thanks Vestislav Apostolov  for insightful discussions and invitations all along the preparation of this article. The author is also grateful to Karl Oeljeklaus for discussions on his work \cite{OT08} in order to obtain the adapted result needed in the main theorem. }

\date{\today}

\begin{abstract} Let $S$ be a minimal compact complex surface with Betti numbers $b_1(S)=1$ and $b_2(S)\ge 1$ i.e. a compact surface in class VII$_0^+$. We show that if there exists a twisted logarithmic 1-form $\t\in H^0(S,\O^1(\log D)\ot \cal L_\l)$, where $D$ is a non zero divisor and $\cal L\in H^1(S,\bb C^\star)$,  then $S$ is a Kato surface. It is known that  $\l$ is in fact real and we show that  $\l\ge 1$ and unique if $S$ is not a Inoue-Hirzebruch surface. Moreover $\l=1$ if and only if $S$ is a Enoki surface. When $\l>1$ these conditions are equivalent to the existence of a negative PSH function $\hat \tau$ on the cyclic covering $p:\hat S\to S$ of $S$ which is PH outside $\hat D:=p^{-1}(D)$ with automorphic constant being the {\bf same} automorphy constant $\l$ for a suitable automorphism of $\hat S$. With previous results obtained with V.Apostolov it suggests a strategy to prove the GSS conjecture.

\end{abstract}

\maketitle


\section{Introduction} The classification of minimal compact complex surfaces is incomplete for non K\"ahler surfaces $S$ in VII$_0^+$ class, that is minimal compact complex surfaces which have Betti numbers $b_1(S)=1$ and $b_2(S)\ge 1$. All known examples of such surfaces contain Global Spherical Shells (GSS) and are called Kato surfaces \cite{K77,D84}. The GSS conjecture asserts that there is no other surfaces. Given a surface $S$ in class VII$_0^+$ several criteria allow us to assert that $S$ is a Kato surface
\begin{itemize}
\item Existence on $S$ of $b_2(S)$ rational curves \cite{DOT03},
\item Existence of a non trivial section of a power of the anticanonical bundle twisted by a topologically trivial line bundle $K_S^{-m}\ot \cal L$, $m\ge 1$, \cite{Dl},
\item Existence of a non trivial global vector field \cite{DOT01},
\item $b_2(S)=1$ by Teleman theorem \cite{andrei1}.
\end{itemize}

In \cite{Brunella14} M.Brunella proved that if the cyclic covering $\pi:\hat S\to S$ of $S$ admits a negative plurisubhamonic (PSH) function $\hat u:\hat S\to [-\infty,0[$ which is pluriharmonic (PH) outside a non empty analytic subset and for a generator $\f:\hat S\to \hat S$, $\hat\tau\circ \f= \l \hat\tau$ for some positive $\l$, then $S$ is a (possibly) blown-up Kato surface.\\
It can be noticed that $\l$ cannot be equal to $1$ and replacing if necessary $\f$ by $\f^{-1}$, we may suppose that $\l>1$. Therefore, by \cite{DO}, $S$ is an intermediate Kato surface (i.e. the maximal divisor $D$ in $S$ is a cycle of rational curves with branches) or a Inoue-Hirzebruch surface.\\
Given a compact complex surface in class VII$_0^+$, the main results (see thm \ref{logformPSH}) of the article are 
\begin{enumerate}
\item The existence of a non trivial twisted logarithmic 1-form 
$$\t\in H^0(S, \O^1(\log D)\ot \cal L_\l),$$
 where $D$ is a non trivial divisor and $\cal L_\l\in H^1(S,\bb C^\star)\simeq \bb C^\star$ implies that $S$ is a Kato surface,
\item The existence of such a twisted logarithmic 1-form $\t$ is equivalent to the existence of a negative PSH function $\hat u:\hat S\to [-\infty,0[$ on the cyclic covering $\pi:\hat S\to S$ of $S$ which is PH outside $\hat D:=\pi^{-1}(D)$ such that for a suitable generator $\g:\hat S\to \hat S$ of the deck group of automorphism of $\hat S$, 
$$\hat u\circ \g=\lambda \hat u.$$
\item If $\l\ge 1$ $\l$ is unique and it is the same constant for both twisted logarithmic 1-form $\t$ and PSH function $\hat u$.
\end{enumerate}

A locally conformally symplectic structure (LCS) tamed by the complex structure is a 2-form $\o$ on $S$ such that $d_\a \o=0$, for a closed 1-form $\a$ and whose $(1,1)$ part $\o^{1,1}$ is positive.
In a sequence of articles \cite{lcs,lcs1,AD23} it is proved that the set of LCS structures tamed by the complex structure of $S\in {\rm VII}_0^+$, denoted by $\cal T(S)$ can be represented by a non empty  open interval in $]-\infty,0[$.\\
 In this article there are two ways to represent a holomorphically trivial line bundle: by closed one forms or by a real number. More precisely, let $\pi:\hat S\to S$ be the cyclic covering of $S$. For $\a$ a closed real 1-form on $S$, let $\hat \a=\pi^\star \a$ which is exact. If $\g:\hat S\to \hat S$ is a generator of the deck group of automorphisms there is a constant $c\in\bb R$ such that for $\hat \a=d\hat f$, $\hat f\circ\g-\hat f=c$. Replacing $\g$ by $\g^{-1}$ if necessary we may suppose that $c<0$. Consider the mapping 
 $$]-\infty,0[\to ]1,+\infty[,\quad  c\mapsto C=e^{-c}.$$
 This mapping gives the correspondance between a line bundle $\cal L_\a$, $\a\in H^1_{dR}(S,\bb R)$, and $\cal L_C$, $C>1$ setting $C=\f(c)$. Then we have
\begin{conj}   If $\cal T(S)=]a,b[$, $b\le 0$, we have $\l=\f(b)=e^{-b}\ge 1$.
\end{conj}
For hyperbolic Kato surfaces, it is shown that $\f(b)\ge \l$ ($\f$ is decreasing) \cite[section 4.3]{AD23}. Moreover we have for Enoki surfaces \cite{enoki} $0=\l=\f(0)$ because in this case $\cal T(S)=]-\infty,0[$.\\
This conjecture gives the expected unique value for which we have to look for a PSH function, PH outside $D$ with automorphy constant $\l$. Therefore a strategy to prove the GSS conjecture would be to show the existence of a negative PSH function, as in  theorem \ref{periodesmemesigne}, with twisting constant being defined by the upper bound of the open interval $\cal T(S)$ .\\
 A sketch of proof is the following: Given a twisted logarithmic 1-form we prove first that the Camacho-Sad indices \cite{CS82} are negative, therefore this assumption in Lemma 6.18 of \cite{AD23} becomes superfluous. We obtain
 \begin{prop} Let $S$ be a minimal surface in the class VII$_0^+$ which is not a Inoue-Hirzebruch  surface. Suppose that $S$ admits a non-trivial logarithmic 1-form
 $$0\neq \t\in H^0(S,\O^1(\log D)\ot \cal L_\a), \quad [\a]\in H^1_{dR}(S,\bb R)$$
 with pole along a non-trivial connected divisor $D$. Then we can multiply $\t$ by a non zero complex number such that the residues $Res_\a(\t)$ are real valued and positive, and the real part ${\rm Re}(P_\t)$ of the current $P_\t$ associated to $\t$ is exact, i.e. satisfy ${\rm Re}(P_\t)=d_\a\tau$ for a degree zero current $\tau$ on $S$. Furthermore,
 $$d_\a d_\a^c\tau=2\pi Res_\a(\t)T_D>0$$
 and $[\a]\not\in \cal T(S)$.
 \end{prop}
 Notice that the support of $d_\a d_\a^c\tau$ is $D$.
 With this proposition we find a PSH function $\hat u=e^{-\hat f}\hat\tau$, PH outside $\hat D$. Conversely, $i\part \hat u$ is holomorphic outside $\hat D$ since $\hat u$ is there PH. The end of the proof consists to extend $i\part \hat u$ on whole $\hat S$ into a logarithmic 1-form.\\

\section{Negativity of Camacho-Sad coefficients}
We start with an explicit example and after we use a family of polynomials, already used in \cite{DOT01} to solve a system of equations which determine the Camacho-Sad indices of the foliation defined  by a global twisted logarithmic 1-form. We want to show that this system of equations has positive solutions. When $p=2$ and $\d=(1,1)$  the solution is not real, therefore we have to use the geometric properties of the foliation.
\subsection{An explicit example}
\begin{example} We use notations and explicit construction of a Kato surface used in \cite{D84}. Consider the sequence $(42 2 3)=(s_2 r_1 s_1)$ of opposite self-intersections. The associate Kato surface has four rational curves of self-intersections 
$$(D_0^2,D_1^2,D_2^2,D_3^2)=(-4,-2,-2,-3).$$
Its opposite  intersection matrix is
$$M(S)=\left( \begin{matrix} 4&0&0&-1\\0&2&-1&-1\\0&-1&2&-1\\-1&-1&-1&3\end{matrix}\right)$$
therefore we have a cycle with three curves $D_1$, $D_2$, $D_3$ and a tree $D_0$ which intersects $D_3$. The normal form of the contraction $F(z_1,z_2)=(\star , z_2^k)$, associated to the corresponding sequence of blow-ups, has a second member $z_2^k$ where 
$$k=k(S)=\sqrt{\det M(S)}+1=\sqrt{9}+1=4$$
by a result of C.Favre (see \cite{DO} or \cite{Fav00} for the normal form of the contraction).\\
We compute now Camacho-Sad indices setting 
$$\a_1=CS(\cal F, D_1, D_1\cap D_{2}), \quad \a_2=CS(\cal F,D_2,D_2\cap D_3), \quad \a_3=CS(\cal F,D_3,D_3\cap D_1)$$
hence we have chosen an orientation of the cycle. We have, by Camacho-Sad formula
$$\begin{array}{ccccl}
-4&=&D_0^2&=&CS(\cal F, D_0, D_0\cap D_3)\\
\\
 -2&=&D_1^2&=&CS(\cal F,D_1,D_1\cap D_2)+CS(\cal F,D_1,D_1\cap D_3)=\dps \a_1+\frac{1}{\a_3}\\
\\
 -2&=&D_2^2&=&CS(\cal F,D_2,D_2\cap D_3)+CS(\cal F,D_2,D_2\cap D_1)=\dps \a_2+\frac{1}{\a_1}\\
\\
-3&=&D_3^2&=&CS(\cal F,D_3,D_3\cap D_1)+CS(\cal F,D_3,D_3\cap D_2)+CS(\cal F,D_3,D_3\cap D_0)\\
\\
&&&=&\dps \a_3+\frac{1}{\a_2}-\frac{1}{4}
\end{array}$$

We obtain a system of equations for the cycle
$$\left\{\begin{array}{ccc}
\dps\a_1+\frac{1}{\a_3}&=&-2\\
\\
\dps \a_2+\frac{1}{\a_1}&=&-2\\
\\
\dps \a_3+\frac{1}{\a_2}&=&\dps -\frac{11}{4}
\end{array}\right.$$
We shall show that Camacho-Sad indices are negative. For conveniency we change all the signs ($\a_i$ is changed into $-\a_i$) and we solve
$$\left\{\begin{array}{ccc}
\dps\a_1+\frac{1}{\a_3}&=&2\\
\\
\dps \a_2+\frac{1}{\a_1}&=&2\\
\\
\dps \a_3+\frac{1}{\a_2}&=&\dps \frac{11}{4}
\end{array}\right.$$
and we shall show that all the solutions are positive. For example look at $\a_1$: it is equal to a periodic continued fraction or is a solution of an equation of degree two:
$$\a_1=2-\frac{1}{\a_3}=2-1/ \frac{11}{4}-1/2-\frac{1}{\a_1}$$
$$6\a_1^2-13\a_1+6=0, \quad \D=25$$
\begin{itemize}
\item If $\a_1=\frac{3}{2}$ then $\a_2=2$, $\a_3=\frac{4}{3}$.\\
\item If $\a_1=\frac{2}{3}$, then $\a_2=\frac{3}{4}$, $\a_3=\frac{1}{2}$.
\end{itemize}
The two sets of solutions are inverse of each other, corresponding of the same foliation but changing the orientation of the cycle. By \cite{DO} there is only one foliation on $S$.\\
We notice that
$$\frac{3}{2}.2.\frac{4}{3}= 4=k(S), \quad \frac{2}{3}.\frac{3}{4}.\frac{1}{2}=\frac{1}{4}=\frac{1}{k(S)}.$$
It gives the multiplicative factor, hence the coefficient of the torsion line bundle of the logarithmic $1$-form. In these relations, only Camacho-Sad indices of intersection points of the cycle come into play and no other singularities. However the values of these indices depend also on the  tree.\\
These relation reveal that, in general, for a root of one of the degree  two equations, its inverse is a root of another degree two equation. Therefore the roots exist by pairs of the same sign.
\end{example}

\subsection{Algebraic part of the proof}
\begin{lem} \label{inversion} We consider the family of polynomials
$$P_0=1,\  P_1(X_1)=X_1,\  P_i(X_1,\ldots,X_i)=X_1P_{i-1}(X_2,\ldots,X_i)-P_{i-2}(X_3,\ldots,X_i),\ i\ge 2$$
Then for any $n\ge 1$, $P_n(X_1,\ldots,X_n)=P_n(X_n,\ldots,X_1)$.
\end{lem}
\begin{proof} Remark that the definition gives
$$P_i(X_i,\ldots,X_1)=X_iP_{i-1}(X_{i-1},\ldots,X_1)-P_{i-2}(X_{i-2},\ldots,X_1).$$
Therefore we have to show by induction on $n\ge 2$ that
$$P_i(X_1,\ldots,X_i)=X_iP_{i-1}(X_{i-1},\ldots,X_1)-P_{i-2}(X_{i-2},\ldots,X_1),$$
It is easy to check that for $n=2,3$ the two expressions give the same polynomials. 
Suppose that it is true for $i$ and $i-1$; then for $i+1$ we have by the induction hypothesis
$$\begin{array}{ccl}
P_{i+1}(X_1,\ldots,X_{i+1})&=&X_1P_i(X_2,\ldots,X_{i+1})-P_{i-1}(X_3,\ldots,X_{i+1})\\
\\
&=&X_1\bigl[ X_{i+1}P_{i-1}(X_i,\ldots,X_2)-P_{i-2}(X_{i-1},\ldots,X_2)\bigr]\\
\\
&&-\bigl[X_{i+1}P_{i-2}(X_i,\ldots,X_3)- P_{i-3}(X_{i-1},\ldots,X_3)\bigr]\\
\\
&=&X_{i+1}\bigl[X_1P_{i-1}(X_i,\ldots,X_2)- P_{i-2}(X_{i},\ldots,X_3)\bigr]\\
\\
&&-\bigl[X_1P_{i-2}(X_{i-1},\ldots,X_2)- P_{i-3}(X_{i-1},\ldots,X_3)\bigr]\\
\\
&=&X_{i+1}P_i(X_i,\ldots,X_1) - P_{i-1}(X_{i-1},\ldots,X_1)

\end{array}$$
\end{proof}

\begin{lem} \label{cyclic} For $p\ge 2$, the polynomials
$$Q_p(X_0,\ldots,X_{p-1}):=P_p(X_0,\ldots,X_{p-1})-P_{p-2}(X_1,\ldots,X_{p-2})$$
 are invariant under cyclic permutation.
\end{lem}
\begin{proof} It is sufficient to prove that 
$$D:=Q_p(X_0,\ldots,X_{p-1})-Q_p(X_1,\ldots,X_{p-1},X_0)=0.$$
 In fact, by lemma \ref{inversion}
$$\begin{array}{ccl}
D&:=&\bigl(P_p(X_0,\ldots,X_{p-1})-P_{p-2}(X_1,\ldots,X_{p-2})\bigr)\\
&&\\
&& - \bigl(P_p(X_1,\ldots,X_{p-1},X_0)-P_{p-2}(X_2,\ldots,X_{p-1})\bigr)\\
&&\\
&=&\bigl(X_0P_{p-1}(X_1,\ldots,X_{p-1})-P_{p-2}(X_2,\ldots,X_{p-1})-P_{p-2}(X_1,\ldots,X_{p-2})\bigr)\\
&&\\
&&-\bigl(X_0P_{p-1}(X_1,\ldots,X_{p-1})-P_{p-2}(X_1,\ldots,X_{p-2})-P_{p-2}(X_2,\ldots,X_{p-1})\bigr)\\
&&\\
&=&0.
\end{array}$$
\end{proof}

\begin{example} $P_0=1$, $P_1(X_0)=X_0$, $P_2(X_0,X_1)=X_0X_1-1$, 
$$P_3(X_0,X_1,X_2)=X_0X_1X_2-X_0-X_2,$$
$$P_4(X_0,X_1,X_2,X_3)=X_0X_1X_2X_3-X_0X_1-X_0X_3-X_2X_3 +1,$$
$$P_5= X_0X_1X_2 X_3X_4 - X_0X_1X_2 - X_2X_3X_4  - X_3X_4X_0  -  X_4X_0X_1+X_0+X_2+X_4$$
$$\begin{array}{ccl}
P_6&=&X_0X_1X_2 X_3X_4X_5\\
&& -X_0X_1X_2X_3 - X_2X_3X_4X_5 - X_3X_4X_5X_0 - X_4X_5X_0X_1 - X_5X_0X_1X_2\\
&& + X_0X_1 + X_2X_3 + X_4X_5 + X_5X_0 + X_0X_3 + X_2X_5 -1
\end{array}$$
\vspace{1cm}
$$Q_2(X_0,X_1)=X_0X_1-2$$
$$Q_3(X_0,X_1,X_2)=X_0X_1X_2-X_0-X_1-X_2$$
$$Q_4(X_0,X_1,X_2,X_3)=X_0X_1X_2X_3-X_0X_1-X_1X_2-X_2X_3-X_3X_0 +2$$
$$\begin{array}{ccl}
Q_5&=&X_0X_1X_2X_3X_4\\
&& - X_0X_1X_2 - X_1X_2X_3 - X_2X_3X_4 - X_3X_4X_0 - X_4X_0X_1\\
&&  + X_0+X_1+X_2+X_3+X_4
\end{array}$$
Notice that new monomials appear in $Q_6$:
$$\begin{array}{ccl}
Q_6&=&X_0X_1X_2X_3X_4X_5\\
&&-X_0X_1X_2X_3-X_1X_2X_3X_4 - X_2X_3X_4X_5 - X_3X_4X_5X_0 - X_4X_5X_0X_1 -X_5X_0X_1X_2\\
&&+X_0X_1+X_1X_2+X_2X_3+X_3X_4+X_4X_5+X_5X_0+\mathbf {X_0X_3+X_1X_4+X_2X_5} -2
\end{array}$$
\end{example}

\begin{lem} \label{positif} If $\d_i$ are real numbers such that $\d_i\ge 2$ for $i\ge 0$, then for any $j\ge 0$,
$$\frac{P_{j+1}(\d_i,\ldots,\d_{i+j})}{P_{j}(\d_i,\ldots,\d_{i+j-1})}\ge \frac{j+2}{j+1}>1, $$

$$P_{j+1}(\d_{i},\ldots,\d_{i+j})\ge j+2,$$
and for $p\ge 2$,
$$P_p(\d_i,\ldots,\d_{i+p-1})-P_{p-2}(\d_{i+1},\ldots,\d_{i+p-2})\ge 2.$$
Moreover, if $\d=(2,\ldots,2)$ then the previous inequalities are equalities.
\end{lem}

\begin{proof} By  induction on $j\ge 0$. 
If $j=0$, $\dps\frac{P_1(\d_{i})}{P_0}\ge 2$. 
Suppose the inequality fullfilled for $j\ge 0$, then
$$\begin{array}{lcl}
\dps\frac{P_{j+2}(\d_{i},\ldots,\d_{i+j+1})}{P_{j+1}(\d_{i},\ldots,\d_{i+j})}& = &\dps\frac{\d_{i+j+1}P_{j+1}(\d_{i},\ldots,\d_{i+j})- P_{j}(\d_{i},\ldots,\d_{i+j-1})}{P_{j+1}(\d_{i},\ldots,\d_{i+j})}\\
\\
&=&\dps\d_{i+j+1}-\dps\frac{P_{j}(\d_{i},\ldots,\d_{i+j-1})}{P_{j+1}(\d_{i},\ldots,\d_{i+j})}\\
\\
&\ge&\dps\d_{i+j+1}-\frac{j+1}{j+2}\ge \dps 2-\frac{j+1}{j+2} = \frac{j+3}{j+2}.
\end{array}$$
Therefore
$$P_{j+1}(\d_i,\ldots,\d_{i+j})=\frac{P_{j+1}(\d_i,\ldots,\d_{i+j})}{P_{j}(\d_i,\ldots,\d_{i+j-1})}\cdots \frac{P_1(\d_i)}{P_0}\ge \frac{j+2}{j+1}\cdots \frac{2}{1}=j+2.$$
Using previous inequalities,
$$\begin{array}{l}
P_p(\d_i,\ldots,\d_{i+p-1})-P_{p-2}(\d_{i+1},\ldots,\d_{i+p-2}) \\
\\
\quad = P_{p-2}(\d_{i+1},\ldots,\d_{i+p-2})\left\{\dps\frac{P_p(\d_i,\ldots,\d_{i+p-1})}{P_{p-1}(\d_i,\ldots,\d_{i+p-2})} \frac{P_{p-1}(\d_i,\ldots,\d_{i+p-2})}{P_{p-2}(\d_{i+1},\ldots,\d_{i+p-2})} -1\right\}\\
\\
\quad \ge (p-1) \left\{\dps \frac{p+1}{p}.\frac{p}{p-1}-1\right\}=2.
\end{array}$$

It can be easily checked by induction on $p\ge 0$ that $P_p(2,\ldots,2)=p+1$.
\end{proof}

\begin{lem} \label{equationdegre2}Let $\a=(\a_0,\ldots,\a_{p-1})$ be a solution (real or complex) of the system
$$\left\{\begin{array}{ccl}
\dps\a_0+\frac{1}{\a_1}&=&\d_0\\
\vdots&&\vdots\\
\dps\a_i+\frac{1}{\a_{i+1}}&=&\d_i\\
\vdots&&\vdots\\
\dps\a_{p-1}+\frac{1}{\a_0}&=&\d_{p-1}
\end{array}\right.\leqno{(E_p)}$$
where $\a_i\neq 0$ and $\d_i$ are real,  for $i=0,\ldots,p-1$.
Then, we have (with $P_{-1}=0$),

$$\a_{i+1}=\frac{\a_0P_{i-1}(\d_1,\ldots,\d_{i-1})-P_{i}(\d_0,\ldots,\d_{i-1})}{\a_0P_i(\d_1,\ldots,\d_i)-P_{i+1}(\d_0,\ldots,\d_i)}\neq 0 \quad i=0,\ldots,p-1,$$
where 
$$\a_0\neq \frac{P_{i+1}(\d_0,\ldots,\d_i)}{P_i(\d_1,\ldots,\d_i)}, \quad i=0,\ldots,p-1.$$

In particular, $\a_0$ satisfies the degree two equation
$$a_0(\d)X^2+b_0(\d)X+c_0(\d)=0  \leqno{(\cal E_0)}$$
where
$$a_0(\d)=P_{p-1}(\d_1,\ldots,\d_{p-1})\neq 0, \quad   c_0(\d)=P_{p-1}(\d_0,\ldots,\d_{p-2})\neq 0$$
$$b_0(\d)=-\Bigl(P_p(\d_0,\ldots,\d_{p-1})+P_{p-2}(\d_1,\ldots,\d_{p-2})\Bigr)$$

By circular permutation we have also similar relations with $\a_j$, $0\le j\le p-1$:\\
 For $i=0,\ldots,p-1$,
$$\a_{j+i+1}=\frac{\a_jP_{i-1}(\d_{j+1},\ldots,\d_{p-1},\d_0,\ldots,\d_{j+i-1})-P_i(\d_j,\ldots,\d_{p-1},\d_0,\ldots,\d_{j+i-1})}{\a_j P_i(\d_{j+1},\ldots,\d_{p-1},\d_0,\ldots,\d_{j+i})-P_{i+1}(\d_j,\ldots,\d_{p-1},\d_0,\ldots,\d_{j+i})},$$
where
$$\a_j\neq \frac{P_{i+1}(\d_j,\ldots,\d_{p-1},\d_0,\ldots,\d_{j+i})}{P_i(\d_{j+1},\ldots,\d_{p-1},\d_0,\ldots,\d_{j+i})}\neq 0, \quad i=0,\ldots,p-1.$$
In particular, $\a_j$ satisfies the degree two equation
$$a_j(\d)X^2+b_j(\d)X+c_j(\d)=0 \leqno{(\cal E_j)}$$
where
$$a_j(\d)=P_{p-1}(\d_{j+1},\ldots,\d_{p-1},\d_0,\ldots,\d_{j-1})\neq 0, \quad   c_j(\d)=P_{p-1}(\d_j,\ldots,\d_{p-1},\d_0,\ldots,\d_{j-2})\neq 0$$
$$b_j(\d)=-\Bigl(P_p(\d_j,\ldots,\d_{p-1},\d_0,\ldots,\ldots,\d_{j-1})+P_{p-2}(\d_{j+1},\ldots,\ldots,\d_{p-1},\d_0,\ldots,\d_{j-2})\Bigr)$$
We have $a_j(\d)=c_{j+1}(\d)\neq 0$ and 
$$\prod_{j=0}^{p-1}P_{p-1}(\d_j,\ldots,\d_{j-2})=\prod_{j=0}^{p-1}a_j(\d)=\prod_{j=0}^{p-1}c_j(\d)\neq 0.$$

Moreover if one root of $(\cal E_j)$ is real then all the roots of all the equations $(\cal E_i)$, $i=0,\ldots,p-1$ are real.
\end{lem}
\begin{rem} For some values of $\d$ the system of equations $(E_p)$ has no solution. For example, for $p=2$  and $\d=(0,1)$.
\end{rem}
\begin{proof} For $i=0$, $\a_1=\dps\frac{P_0}{P_1(\d_0)-\a_0P_0}$ and $\a_0\neq \dps\frac{P_1(\d_0)}{P_0}$.\\
As 
$$\a_2=\dps\frac{P_1(\d_0)-\a_0P_0}{P_2(\d_0,\d_1)-\a_0P_1(\d_1)}$$
it suggests that for an integer $i\ge 0$,
$$\a_{i+1}=\frac{P_{i}(\d_0,\ldots,\d_{i-1})-\a_0P_{i-1}(\d_1,\ldots,\d_{i-1})}{P_{i+1}(\d_0,\ldots,\d_i)-\a_0P_i(\d_1,\ldots,\d_i)}$$
with
$$\a_0\neq \frac{P_{j+1}(\d_0,\ldots,\d_j)}{P_j(\d_1,\ldots,\d_j)},\quad  j=0,\ldots,i.$$
and this will be proved by induction. \\
From $\dps\frac{1}{\a_{i+2}}=\d_{i+1}-\a_{i+1}$ we deduce
$$\begin{array}{lcl}
\a_{i+2}&=&\dps\frac{1}{\d_{i+1}-\dps\frac{P_i(\d_0,\ldots,\d_{i-1})-\a_0P_{i-1}(\d_1,\ldots,\d_{i-1})}{P_{i+1}(\d_0,\ldots,\d_i)-\a_0P_i(\d_1,\ldots,\d_i)}} \\
\\
&=&\dps \frac{P_{i+1}(\d_0,\ldots,\d_i)-\a_0P_i(\d_1,\ldots,\d_i)}{P_{i+2}(\d_0,\ldots,\d_{i+1})-\a_0P_{i+1}(\d_1,\ldots,\d_{i+1})}
\end{array}$$
and $\a_0\neq \dps\frac{P_{i+2}(\d_0,\ldots,\d_{i+1})}{P_{i+1}(\d_1,\ldots,\d_{i+1})}$. \\

For $i=p-1$ we have
$$\a_0=\frac{\a_0P_{p-2}(\d_1,\ldots,\d_{p-2}-P_{p-1}(\d_0,\ldots,\d_{p-2})}{\a_0P_{p-1}(\d_1,\ldots,\d_{p-1})-P_p(\d_0,\ldots,\d_{p-1})}$$
which gives the announced degree two equation satisfied by $\a_0$.\\
The similar formulas with $\a_j$ are clear. Since $a_j(\d)=c_{j+1}(\d)$, $j\ {\rm mod}\ p$,
$$\prod_{j=0}^{p-1}a_j(\d)=\prod_{j=0}^{p-1}c_j(\d)$$
and solution of the system $(E_p)$ exist \iff $c_j(\d)\neq 0$ for $j=0,\ldots,p-1$.
\end{proof} 
\begin{lem} \label{discriminant} The polynomial discriminant
 $$\begin{array}{ccl}
 \D_p(X_0,\ldots,X_{p-1})&:=&\bigl(P_p(X_0,\ldots,X_{p-1}) +P_{p-2}(X_1,\ldots,X_{p-2})\bigr)^2\\
 &&\\
 && - 4P_{p-1}(X_0,\ldots,X_{p-2})P_{p-1}(X_1,\ldots,X_{p-1})
 \end{array}$$
 satisfies the equality
 $$ \D_p(X_0,\ldots,X_{p-1})=\bigl(P_p(X_0,\ldots,X_{p-1}) - P_{p-2}(X_1,\ldots,X_{p-2})\bigr)^2-4$$
 and  is invariant by circular permutation.
\end{lem}

\begin{proof} 1) Let $0\le i\le p-1$ and
$$\begin{array}{lcl}
A_p(X_i,\ldots,X_{p-1},X_0,\ldots,X_{i-1})&:=&P_p(X_i,\ldots,X_{i-1})P_{p-2}(X_{i+1},\ldots,X_{i-2}) \\
\\
&&- P_{p-1}(X_i,\ldots,X_{i-2})P_{p-1}(X_{i+1},\ldots,X_{i-1}).
\end{array}$$
Developping from both sides, thanks to the definition of these polynomials and lemma \ref{inversion}, we have
$$\begin{array}{l}
A_p(X_i,\ldots,X_{i-1})\\
\\
= \Bigl\{X_{i-1}X_iP_{p-2}(X_{i+1},\ldots,X_{i-2})^2 -X_iP_{p-3}(X_{i+1},\ldots,X_{i-3})P_{p-2}(X_{i+1},\ldots,X_{i-2}) \\
\quad - P_{p-2}(X_{i+2},\ldots,X_{i-1}) P_{p-2}(X_{i+1},\ldots,X_{i-2})\Bigr\}\\
\\
\quad -\Bigl\{X_{i-1}X_iP_{p-2}(X_{i+1}\ldots,X_{i-2})^2 -X_iP_{p-3}(X_{i+1},\ldots,X_{i-3})P_{p-2}(X_{i+1},\ldots,X_{i-2}) \\
\quad - X_{i-1}P_{p-2}(X_{i+1},\ldots,X_{i-2})P_{p-3}(X_{i+2},\ldots,X_{i-2})\\
\quad +P_{p-3}(X_{i+2},\ldots,X_{i-2})P_{p-3}(X_{i+1},\ldots,X_{i-3})\Bigr\}\\
\\
=\Bigl[X_{i-1}P_{p-2}(X_{i+1},\ldots,X_{i-2})-P_{p-3}(X_{i+1},\ldots,X_{i-3})\Bigr]P_{p-3}(X_{i+2},\ldots,X_{i-2})\\
\\
\quad -P_{p-2}(X_{i+1},\ldots,X_{i-2})P_{p-2}(X_{i+2},\ldots,X_{i-1})\\
\\
=P_{p-1}(X_{i+1},\ldots,X_{i-1})P_{p-3}(X_{i+2},\ldots,X_{i-2}) 

- P_{p-2}(X_{i+1},\ldots,X_{i-2})P_{p-2}(X_{i+2},\ldots,X_{i-1})\\
\\
=A_{p-1}(X_{i+1},\ldots,X_{i-1})
\end{array}$$
Therefore, 
$$\begin{array}{lcl}
A_p(X_i,\ldots,X_{i-1})&=&A_{p-1}(X_{i+1},\ldots,X_{i-1})=\cdots = A_2(X_{i-2},X_{i-1})\\
&&\\
&=&P_2(X_{i-2},X_{i-1})P_0-P_1(X_{i-2})P_1(X_{i-1})=-1.
\end{array}$$
2) We have by 1),
$$\D_p(X_0,\ldots,X_{p-1})=\bigl(P_p(X_0,\ldots,X_{p-1}) -P_{p-2}(X_1,\ldots,X_{p-2})\bigr)^2-4$$
By lemma \ref{cyclic}, $\D_p$ is invariant by cyclic permutation.
\end{proof}

\begin{lem} For $q=(q_0,\ldots,q_{j-1})$ where $q_i=0$ or $q_i=1$, $i=0,\ldots,j-1$,  the partial derivatives of the polynomials $P_j$ satisfy the relation
$$\frac{\part^{|q|}P_j(X_0,\ldots,X_{j-1})}{\part^{q_0}X_0\cdots \part^{q_{p-1}}X_{j-1}}=P_{j-|q|}(X_0^{\e_0},\ldots,X_{j-1}^{\e_{j-1}})$$
where 
$$\e_i=\left\{\begin{array}{ccc}0&{\rm if}&q_i=1\\ 1&{\rm if}&q_i=0\end{array}\right..$$
The condition $\e_i=0$ means that $X_i$ is omitted.\\
In particular,  by lemma \ref{positif}, for $q^1=(q^1_0,\ldots,q^1_{p-1})$ and $q^2=(q^2_0,\ldots,q^2_{p-1})$ such that $|q^1|=|q^2|$,
$$\frac{\part^{|q^1|}P_j(X_0,\ldots,X_{p-1})}{\part^{q^1_0}X_0\cdots \part^{q^1_{p-1}}X_{p-1}}(2,\ldots,2)=\frac{\part^{|q^2|}P_j(X_0,\ldots,X_{p-1})}{\part^{q^2_0}X_0\cdots \part^{q^2_{p-1}}X_{p-1}}(2,\ldots,2)=P_{j-|q_1|}(2,\ldots,2)=j-|q_1|+1.$$

\end{lem}
\begin{proof} By induction on $j\ge 1$ for $q=(q_0,\ldots,q_{j-1})$ such that $|q|\le j$.\\
The relation is trivial for $j=1$. Suppose it is true for an integer $j\ge 1$, then by induction hypothesis,
$$\begin{array}{ccl}
\dps\frac{\part^{|q|}P_{j+1}(X_0,\ldots,X_j)}{\part^{q_0}X_0\cdots\part^{q_j}X_j}&=&\dps\frac{\part^{|q|}}{\part^{q_0}X_0\cdots\part^{q_j}X_j}\bigl(X_0P_j(X_1,\ldots,X_j)-P_{j-1}(X_2,\ldots,X_j)\bigr)\\
&&\\
&=&\left\{
\begin{array}{lcc}\dps\frac{\part^{|q|-1}P_{j}(X_1,\ldots,X_j)}{\part^{q_1}X_1\cdots\part^{q_j}X_j}=P_{j-|q|+1}(X_1^{\e_1},\ldots,X_j^{\e_j})&{\rm if}&q_0=1\\
&&\\
X_0\dps\frac{\part^{|q|}P_{j}(X_1,\ldots,X_j)}{\part^{q_1}X_1\cdots\part^{q_j}X_j}- \frac{\part^{|q|}P_{j-1}(X_1,\ldots,X_j)}{\part^{q_1}X_1\cdots\part^{q_j}X_j}&{\rm if}&q_0=0\\
\end{array}\right.\\
&&\\
&=&P_{j-|q|+1}(X_0^{\e_0},X_1^{\e_1},\ldots,X_j^{\e_j}).
\end{array}$$

\end{proof}

\subsection{Geometric part of the proof}
If the surface $S$ admits a twisted logarithmic 1-form it is known \cite[lemma 6.12]{AD23} that $S$ contains a cycle of rational curves. Moreover by \cite{D21} extra curves are composed of branches (called also trees), each curve of the cycle meeting at most one branch.

\begin{lem}\label{onlysing} The foliation $\cal F$ associated to a non trivial logarithmic $1$-form 
$$\theta\in H^0(S,\O^1(\log D)\ot \cal L_\lambda)$$
 has no other singularities on $D$ than the intersection points of rational curves.
\end{lem}
\begin{proof} Let $p$ be a regular point of $D$. Let $\Pi:\hat S\to S$ be the cyclic covering and we set $\hat D=\Pi^{-1}(D)$, $\hat\a:=\Pi^\star\a$, $\hat\a=d\hat h$, $\hat\omlog=\Pi^\star\omlog$. In a local chart in a neighbourhood of  $p\in D_j=\{z_2=0\}$,
$$e^{\hat h}\hat\omlog=\omlog_0+g_2(z)\frac{dz_2}{z_2}$$
where $\omlog_0$ and  $g_2$ are holomorphic. Let $\g_1$ be a closed curve in $\{z_1=c\}$ around $(c,0)$. Since $e^{\hat h}\hat\omlog$ is closed in $\hat S\setminus\hat D$,
$$\int_{\g_1}e^{\hat h}\hat\omlog=\int_{\g_1}g_2(z)\frac{dz_2}{z_2}=   2\pi i Res_{z_2=0}(e^{\hat h}\hat\omlog_{\mid z_1=c})=2\pi ig_2(c,0)$$
does not depend on $c$, therefore $g_2$ depends only on $z_2$.  Hence
$$e^{\hat h}\hat\omlog=\omlog'_0+g_2(0)\frac{dz_2}{z_2}$$
for $\omlog'_0$ holomorphic and $g_2(0)\neq 0$. As $\cal F$ is also defined by the holomorphic form
$$\o=z_2\theta'_0 +g_2(0)dz_2$$
which has no singularity at $p=(0,0)$, $p$ is a regular point of $\cal F$.\\
\end{proof}

\begin{lem} \label{coeffarbre} Let $S$ be in class VII$_0^+$ endowed with a twisted logarithmic $1$-form $\t\in H^0(S,\O^1(\log(D)\ot\cal L_\l)$ with a cycle of rational curves $\G=D_0+\cdots+D_{p-1}$, and denote by $\cal F$ the associated foliation. Let $\a=(\a_0,\ldots,\a_{p-1})$ be a real solution of the system
$$(E_p) \left\{\begin{array}{ccl}
\dps\a_0+\frac{1}{\a_1}&=&\d_0\\
\vdots&&\vdots\\
\dps\a_i+\frac{1}{\a_{i+1}}&=&\d_i\\
\vdots&&\vdots\\
\dps\a_{p-1}+\frac{1}{\a_0}&=&\d_{p-1}
\end{array}\right.$$
where $\a_i=-CS(\cal F,D_i,D_i\cap D_{i-1})$, $i\ {\rm mod}\ p$.
If $A_j=C_0+\cdots+C_{k-1}$ is a chain meeting the cycle $\G$ at a point $A_j\cap D_j=C_{k-1}\cap D_j$ then
$$-\frac{k}{k+1}\le CS(\cal F,D_j,D_j\cap C_{k-1})\le \frac{1}{C_{k-1}^2}.$$
For indices $i$ corresponding of curves of self-intersection $-2$ and with a tree, we have $1<\d_i<2$. For all other indices $\d_i\ge 2$.
\end{lem}
\begin{proof}
Let $A_j=C_0+\cdots+C_{k-1}$, a chain, where $C_0$ is its top, which meets the curve $D_j$ of the cycle,  By lemma \ref{onlysing}, the intersection points of the rational curves are the only singularities of $\mathcal F$, therefore Camacho-Sad formula for $C_0$ is
 $$CS(\cal F,C_0,C_0\cap C_1)= C_0^2 \le -2,$$
with notations of \cite[lemma 6.14]{AD23}.
By induction we shall show that $CS(\cal F,C_i,C_i\cap C_{i+1})$ is a negative rational which satisfies
$$C_i^2\le CS(\cal F,C_i,C_i\cap C_{i+1})\le -\frac{i+2}{i+1}.$$

In fact, it is true for $i=0$.  Camacho-Sad formula for $C_{i+1}$ is
$$C_{i+1}^2=CS(\mathcal F,C_{i+1},C_{i+1}\cap C_i) + CS(\mathcal F,C_{i+1},C_{i+1}\cap C_{i+2}).$$
We recall that $CS(\cal F,C_{i+1},C_i\cap C_{i+1})=CS(\cal F,C_i,C_i\cap C_{i+1})^{-1}$. By induction hypothesis, 
$$\begin{array}{lcl}
C_{i+1}^2&\le& CS(\cal F, C_{i+1}, C_{i+1}\cap C_{i+2}) = C_{i+1}^2 - CS(\cal F, C_{i+1}, C_{i+1}\cap C_{i})\\
&&\\
 &\le&\dps -2 + \frac{i+1}{i+2} = -\frac{i+3}{i+2}
\end{array}$$
For $i=k-1$, we obtain
 $$C_{k-1}^2\le CS(\cal F,C_{k-1},C_{k-1}\cap D_{j}) \le -\frac{k+1}{k}$$
 hence 
 $$-\frac{k}{k+1}\le CS(\cal F,D_{j},C_{k-1}\cap D_{j})\le \frac{1}{C_{k-1}^2}<0.$$
 \end{proof}

\begin{cor} \label{racines3} Let $S$ be a surface in class VII$_0^+$ endowed with a twisted logarithmic $1$-form $\t\in H^0(S,\O^1(\log D)\ot \cal L_\l)$. If each tree meets a curve with self-intersection at most $-3$ all Camacho-Sad indices are negative.
\end{cor}
\begin{proof} There exists a non-trivial twisted logarithmic $1$-form, therefore $S$ contains a cycle of rational curves, hence the maximal divisor of $S$ is a cycle with chains of rational curves (called trees), each chain meeting the cycle at exactly one point and each curve of the cycle meets at most one chain (see \cite{D21}). By lemma \ref{coeffarbre} and Camacho-Sad formula,
$$D_j^2=CS(\cal F,D_j,D_j\cap D_{j-1})+CS(\cal F,D_j,D_j\cap D_{j+1})+CS(\cal F,D_j,D_j\cap C_{k-1})$$
with $0<-CS(\cal F,D_j,D_j\cap C_{k-1})<1$,
therefore, for $\a_j=-CS(\cal F,D_j,D_j\cap D_{j-1})$
$$\a_j+\frac{1}{\a_{j+1}}=\d_j=-D_j^2+CS(\cal F,D_j,D_j\cap C_{k-1})>2.$$
as self-intersections of the curves are less than $-2$. We conclude by lemma \ref{positif} that coefficients $a_j(\d)>0$, $-b_j(\d)>0$, $c_j(\d)>0$. Since product and sum of roots are positive, roots are positive.
\end{proof}

\begin{lem} \label{identite} Let $p\ge 1$ and $x_0,\ldots,x_{p-1}\in\bb R$. Then we have the identity
$$\sum_{j=0}^p(p+1-j)\sum_{q\,\mid\, |q|=j}x_0^{q_0}\cdots x_{p-1}^{q_{p-1}}=\prod_{i=0}^{p-1}(1+x_i)+\sum_{k=0}^{p-1}\prod_{i\neq k}(1+x_i).$$
\end{lem}
\begin{proof}  The relation is clear for $p=1$. If it is true for an integer $p\ge 1$ we have
$$\begin{array}{l}
\dps\sum_{j=0}^{p+1}(p+2-j)\sum_{q\,\mid\,|q|=j}x_0^{q_0}\cdots x_p^{q_p}=\dps\sum_{j=0}^{p+1}\sum_{q\,\mid\,|q|=j}x_0^{q_0}\cdots x_p^{q_p} + \sum_{j=0}^{p}(p+1-j)\sum_{q\,\mid\,|q|=j}x_0^{q_0}\cdots x_p^{q_p}\\
\\
\hspace{1cm} =\dps\prod_{j=0}^p(1+x_j) + \dps\sum_{j=0}^p(p+1-j)\Bigl(\sum_{|q|=j\atop{q_p=1}}x_0^{q_0}\cdots x_p^{q_p}+ \sum_{|q|=j\atop{q_p=0}}x_0^{q_0}\cdots x_{p-1}^{q_{p-1}}\Bigr)\\
\\
\hspace{1cm}=\dps\prod_{j=0}^p(1+x_j) \dps+x_p\sum_{j=1}^p(p+1-j)\sum_{|q'|=j-1}x_0^{q_0}\cdots x_{p-1}^{q_{p-1}} + \sum_{j=0}^p(p+1-j)\sum_{|q'|=j}x_0^{q_0}\cdots x_{p-1}^{q_{p-1}}\\
\hspace{1cm}=\dps\prod_{j=0}^p(1+x_j) +x_p\sum_{j'=0}^{p-1}(p-j')\sum_{|q'|=j'}x_0^{q_0}\cdots x_{p-1}^{q_{p-1}} +  \prod_{j=0}^{p-1}(1+x_j) + \sum_{k=0}^{p-1}\, \prod_{i\neq k}(1+x_i)\\
\hspace{1cm}=\dps\prod_{j=0}^p(1+x_j) +x_p\sum_{j'=0}^{p}(p+1-j')\sum_{|q'|=j'}x_0^{q_0}\cdots x_{p-1}^{q_{p-1}} - x_p\sum_{|q'|\le p}x_0^{q_0}\cdots x_{p-1}^{q_{p-1}}\\
\hspace{15mm}+ \dps \prod_{j=0}^{p-1}(1+x_j) + \sum_{k=0}^{p-1}\, \prod_{i\neq k}(1+x_i)\\
\hspace{1cm}=\dps\prod_{j=0}^p(1+x_j) +x_p \left\{\prod_{j=0}^{p-1}(1+x_j)+\sum_{k=0}^{p-1}\ \prod_{i\neq k}(1+x_i) \right\}+\sum_{k=0}^{p-1}\prod_{i\neq k}(1+x_i)\\
\hspace{1cm}=\dps\prod_{j=0}^p(1+x_j)  +(1+x_p)\sum_{k=0}^{p-1}\ \prod_{i\neq k}(1+x_i)+\prod_{j=0}^{p-1}(1+x_j)\\
\hspace{1cm}=\dps\prod_{j=0}^p(1+x_j) +\sum_{k=0}^p\ \prod_{i\neq k}(1+x_i).
\end{array}$$

\end{proof}

\begin{lem} \label{discriminant>2} Let $\d=(\d_0,\ldots,\d_{p-1})$ be the second member of a system $(E_p)$ which corresponds to a surface $S\in VII_0$ with a cycle of rational curves, in particular admits real solutions. \\ 
Then for any index $0\le i\le p-1$,
$$P_p(\d_i,\ldots,\d_{i-1})-P_{p-2}(\d_{i+1},\ldots,\d_{i-2})>0,$$
$$P_{p-1}(\d_i,\ldots,\d_{i-2})>0$$
\end{lem}
\begin{proof} By lemma \ref{coeffarbre}, $\d_i>1$ for $i=0,\ldots,p-1$. The highest possible power of an indeterminacy in a monomial is one. Therefore the only non vanishing partial derivatives 
$$\frac{\part^{|q|}\bigl(P_p(X_0,\ldots,X_{p-1})-P_{p-2}(X_1,\ldots,X_{p-2})\bigr)}{\part^{q_0}X_0\cdots \part^{q_{p-1}}X_{p-1}}$$
are those for which $q_i=0$, or $q_i=1$, $i=0,\ldots,p-1$, where $|q|=\sum_{i=0}^{p-1}q_i$.  Besides, for
$$Q_p(X_0,\ldots,X_{p-1}):=P_p(X_0,\ldots,X_{p-1})-P_{p-2}(X_1,\ldots,X_{p-2})$$
and fixed $|q|$, $0\le |q|\le p$, we have by lemma \ref{positif},
$$
\dps\frac{\part^{|q|}Q_p(X_0,\ldots,X_{p-1})}{\part^{q_0}X_0\cdots \part^{q_{p-1}}X_{p-1}}(2,\ldots,2)=P_{p-|q|}(2,\ldots,2)-P_{p-2-|q|}(2,\ldots,2)=2.$$
By Taylor formula for polynomials,
$$\begin{array}{l}
Q_p( \d_0,\ldots,\d_{p-1})\\
\\
\hspace{1cm}
=\dps\sum_{0\le j\le p}\ \sum_{q,|q|=j}(\d_0-2)^{q_0}\cdots(\d_{p-1}-2)^{q_{p-1}}\frac{\part^{|q|}Q_p(X_0,\ldots,X_{p-1})}{\part^{q_0}X_0\cdots \part^{q_{p-1}}X_{p-1}}(2,\ldots,2)\\
\\
\hspace{1cm}=\dps\sum_{0\le j\le p}\ \sum_{q,|q|=j}(\d_0-2)^{q_0}\cdots(\d_{p-1}-2)^{q_{p-1}} Q_{p-|q|}(2,\ldots,2)\\
\\
\hspace{1cm}=\dps\ 2 \sum_{0\le j\le p}\ \sum_{q,|q|=j}(\d_0-2)^{q_0}\cdots(\d_{p-1}-2)^{q_{p-1}}\\
\\
\hspace{1cm}= 2\dps \prod_{j=0}^{p-1}\bigl(1+(\d_j-2)\bigr) = 2\prod_{j=0}^{p-1}(\d_j-1)>0.
\end{array}$$
Hence
$$P_p(\d_i,\ldots,\d_{i-1})-P_{p-2}(\d_{i+1},\ldots,\d_{i-2})> 0.$$

In a similar way, for $\d=(\d_0,\ldots,\d_{p-1})$, $\d_i>1$, $i=0,\ldots,p-1$,
$$\begin{array}{lcl}
P_{p-1}(\d_0,\ldots,\d_{p-2})&=&\dps\sum_{0\le j\le {p-1}}\ \sum_{q,|q|=j}(\d_0-2)^{q_0}\cdots(\d_{p-1}-2)^{q_{p-1}}\frac{\part^{|q|}P_{p-1}(X_0,\ldots,X_{p-2})}{\part^{q_0}X_0\cdots \part^{q_{p-1}}X_{p-1}}(2,\ldots,2)\\
&&\\
&=&\dps \sum_{0\le j\le p-1}(p-j)\sum_{|q|=j}(\d_0-2)^{q_0}\cdots(\d_{p-2}-2)^{q_{p-2}}\\
$$\\
&=&\dps\prod_{i=0}^{p-2}(\d_i-1)+\sum_{k=0}^{p-2}\ \prod_{i\neq k}(\d_i-1)>0
\end{array}$$
We have the result by circular permutation.

\end{proof}

\begin{prop} \label{coeffpositifs} Let $\a=(\a_0,\ldots,\a_{p-1})$ be a real solution of the system
$$(E_p)\left\{\begin{array}{ccl}
\dps\a_0+\frac{1}{\a_1}&=&\d_0\\
\vdots&&\vdots\\
\dps\a_i+\frac{1}{\a_{i+1}}&=&\d_i\\
\vdots&&\vdots\\
\dps\a_{p-1}+\frac{1}{\a_0}&=&\d_{p-1}
\end{array}\right.$$
where $\a_i\neq 0$ and $\d_i>1$ are real,  for $i=0,\ldots,p-1$.
If the system $(E_p)$ is associated to a surface $S$ in class VII$_0^+$ endowed with a twisted logarithmic $1$-form then for $i=0,\ldots,p-1$, $\a_i$ satisfies a degree two equation
$$a_i(\d)\a_i^2+b_i(\d)\a_i+c_i(\d)=0$$
where
$$\dps\frac{c_i(\d)}{a_i(\d)}>0 \quad {\rm and}\quad  -\dps\frac{b_i(\d)}{a_i(\d)}>0.$$
In particular $\a_j$ are positive and Camacho-Sad indices are negative.
\end{prop}
\begin{proof} The discriminant of all the equations is by \ref{discriminant}
$$ \D_p(X_0,\ldots,X_{p-1})=\bigl(P_p(X_0,\ldots,X_{p-1}) - P_{p-2}(X_1,\ldots,X_{p-2})\bigr)^2-4$$
Camacho-Sad indices are real by lemma 6.17 of \cite{AD23} therefore $\D_p(\d)\ge 0$, therefore
$$|P_p(\d)-P_{p-2}(\d)|\ge 2.$$
As by lemma \ref{discriminant>2}, $P_p(\d)-P_{p-2}(\d)>0$, we have
$$P_p(\d)-P_{p-2}(\d)\ge 2$$
From the expressions of $a_j$, $b_j$ and $c_j$, $a_j(\d)>0$, $-b_j(\d)>0$ and $c_j(\d)>0$, the roots are real, their sum and product are positive hence $\a_i$ are positive and Camacho-Sad indices are negative.
\end{proof}

\begin{lem} \label{CSnegatif} Let $\t\in H^0(S,\O^1(\log D)\ot \cal L_\mu)$. Denote by $\G=D_0+\cdots+D_{p-1}$  the cycle in $S$.  If $\a_i=-CS(\cal F, D_i,D_i\cap D_{i+1})$, $i$ mod $p$, then
$$\mu=\prod_{i=0}^{p-1}\a_i>0.$$
\end{lem}
\begin{proof} Denote $a_i:=\int_{\g_i}e^{\hat h}\hat\t$ where $\g_i$ is a closed curve around a lift of $D_i$ in $\hat S$. Let $\hat g:\hat S\to \hat S$ be a generator of the automorphism group of $Aut_S(\hat S)$ and $\g_q=\hat g\circ \g_0$, we have
$$\int_{\g_q}e^{\hat h}\hat\t = \int_{\g_0}\hat g^\star e^{\hat h}\hat\t =\mu \int_{\g_0} e^{\hat h}\hat\t $$
 hence
 $$\mu=\frac{\int_{\g_q}e^{\hat h}\hat\t }{\int_{\g_0} e^{\hat h}\hat\t } = \frac{\int_{\g_q}e^{\hat h}\hat\t }{\int_{\g_{q-1}} e^{\hat h}\hat\t }\cdots \frac{\int_{\g_i}e^{\hat h}\hat\t }{\int_{\g_{i-1}} e^{\hat h}\hat\t } \cdots \frac{\int_{\g_1}e^{\hat h}\hat\t }{\int_{\g_{0}} e^{\hat h}\hat\t } = \frac{a_q}{a_{q-1}}\cdots \frac{a_i}{a_{i-1}}
\cdots \frac{a_1}{a_{0}} = \prod_{i=0}^{q-1}\a_i>0$$
By proposition \ref{coeffpositifs}, $\a_i>0$ for $i=0,\ldots,q-1$.
\end{proof}
The following theorem summarises the previous results:
\begin{thm}\label{periodesmemesigne} Let $S$ be a surface in the class VII$_0^+$,  endowed with a non-trivial  logarithmic $1$-form $0\neq \omlog\in H^0(S,\O^1(\log D)\ot \cal L_\mu)$  with poles along a non-trivial divisor $D$. Then
\begin{itemize}
 \item All Camacho-Sad indices of the associated foliation $\cal F$ are negative, 
 \item The coefficient $\mu$ is a positive real number,
  \item We have $\mu\ge 1$ for a suitable choice of the generator of $Aut_S(\hat S)$ and $\mu=1$ \iff
$S$ is a Enoki surface.  \end{itemize}
 \end{thm}
\begin{proof} Proposition  \ref{coeffpositifs} and lemma \ref{CSnegatif}.
\end{proof}

\begin{rem}  If $\d=(2,\ldots2)$ discriminants vanish therefore $\prod \a_i=\prod \frac{1}{\a_i}$ hence $\mu=\mu^{-1}=1$  we have a non-trivial logarithmic 1-form $\theta\in H^0(S,\O^1(\log D))$. We derive  that $S$ is a Enoki surface.
\end{rem}
By Lemma \ref{CSnegatif}, Lemma 6.18 of \cite{AD23} is improved: the hypothesis ``characteristic numbers at the intersection of curves are negative'' is superfluous. More precisely

\begin{prop} \label{lemma6.18AD23} Let $S$ be a minimal surface in class VII$_0^+$, which is neither Enoki nor Inoue-Hirzebruch, endowed with a non trivial logarithmic $1$-form
$$0\neq \theta\in H^0(S,\O^1(\log D)\ot\mathcal L_\alpha), [\alpha]\in H^1_{dR}(S,\bb R),$$
then there exists a degree zero current $\tau$ such that 
$$d_\alpha d_\alpha^c \tau =2\pi Res_\a(\theta) T_D>0.$$
We have therefore on the cyclic covering $\hat S$ of $S$ a PSH function $\hat\tau$, PH outside $\hat D$.
\end{prop}

\section{Torsion of intermediate Kato surfaces} 
Intermediate Kato surfaces are surfaces containing a global spherical shell (GSS) whoser maximal divisor $D$ is a cycle of rational curves with at least one tree (or branch). See \cite{D84,DO} for their properties. A contraction associated to these surfaces has the normal form
$$F(z_1,z_2)=(\ast, z_2^k), $$
where $k=k(S)\ge 2$ is an integer. This integer gives the twisting line bundle of a logarithmic 1-form
$$\t\in H^0(S,\O^1(\log D)\ot \cal L^k)$$
and the automorphy constant of a PSH function $\tilde\tau(z_1,z_2)=\log|z_2]$.\\
Besides, there is a unique smaller integer $m\ge 1$ such that there exists a numerically $m$-anticanonical divisor \cite{Dl} i.e. a divisor $D_m$ such that
$$mK_S+F+[D_m]=0$$
for a flat line $F$. Equivalently $m\ge 1$ is the least integer for which exists $\kappa\in\bb C^\star$ such that $H^0(S,K_S^{-m}\ot \cal L_\kappa)\neq 0$.\\
In the following proposition we give a construction which associates to any intermediate Kato surface $S$ another Kato surface $S'$ of index one. The aim is to show that they have the same torsion
$$k(S)=k(S').$$

In the following proposition ``special'' means that the surface $S$ has $b_2(S)$ rational curves. In \cite{DOT03} it is proved that a surface is special \iff it is a Kato surface.\\

\begin{prop}[\cite{DOT03}p286] \label{RevRam} Let $S$ be a special surface of intermediate type  
with index $m:=m(S)>1$.  Then there exists a diagram

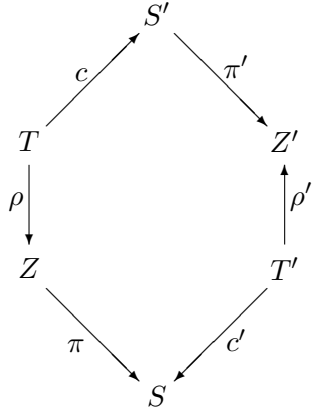
\DIAGV{60}
 {}\n{}\n{S'}\nn
 {}\n{\Near{c}}\n{}\n{\Sear{\pi'}}\nn
{T }\n{}\n{}\n{}\n{ Z'}\nn
{\Sar{\rho}}\n{}\n{}\n{}\n{\naR{\rho'}}\nn
{Z }\n{}\n{}\n{}\n{ T' }\nn
{}\n{\seaR{\pi}}\n{}\n{\swaR{c'}}\nn
{}\n{}\n{S}
\diag

where
\begin{enumerate}
\item $(Z,\pi,S)$ is 
a m-fold cyclic ramified covering space  of $S$, branched over $D$,  endowed with an automorphism group isomorphic to $\bb U_m$ which acts transitively on the fibers,
\item  $(T,\rho,Z)$ is 
 the minimal desingularization of $Z$,
\item $(T,c,S')$ is the contraction of the (possible) exceptional 
curves of the first kind,
\item $S'$ is a special surface with $b_2(S')>0$
with action of the group of $m$-th roots of unity
$\mathbb U_m$, with index $ m(S')=1$,
\item $(S',\pi',Z')$ is the quotient space of $S'$ by $\mathbb U_m$, 
\item $(T',\rho',Z')$ is the minimal desingularization of $Z'$,
\item $(T',c',S)$ is the contraction of the (possible) 
exceptional curves of the first kind,
\end{enumerate}
such that the restriction over $S\setminus D$ is commutative, i.e.
$$\theta:=\pi\circ \rho\circ c^{-1}
=c' \circ {\rho'}^{-1} \circ \pi' : S'\setminus D' \to S\setminus D$$
 and $(S'\setminus D', \theta,S\setminus D)$ is a $m$-fold non ramified covering. Moreover 
\begin{itemize}
\item $S'$ has a GSS \iff $S$ has a GSS,
\item The maximal divisors $D$ and $D'$ of $S$ and $S'$ respectively have the same number of cycles and branches.
\end{itemize}
\end{prop}

We have the following proposition of Oeljeklaus-Toma in \cite[p330-331]{OT08}:
\begin{prop}[\cite{OT08}] \label{kOT} The two intermediate Kato surfaces $S$ and $S'$ of the diagram of proposition \ref{RevRam} have the same torsion of logarithmic $1$-forms i.e.
$$k(S)=k(S').$$
\end{prop}
\begin{proof} Following the diagram of proposition \ref{RevRam} 
$$c'\circ{\rho'}^{-1}\circ \pi':S'\setminus D'\to S\setminus D$$
 is a Galois covering with fiber $\bb Z/m\bb Z$. By Remark 4.4. in \cite{OT08} every surface of index $m\ge 2$ is obtained in this way. Recall that any intermediate Kato surface $S$ is defined by a contracting germ \cite{Fav00} in normal form
$$F(z,\zeta)=(\lambda \zeta^s\, z+P(\zeta)+c\,\zeta^{\frac{sk}{k-1}},\zeta^k)$$
where 
$$k,s\in\bb Z,\quad k>1, \quad s>0, \quad \lambda\in\bb C^\star,$$
$P$ is a complex polynomial,
$$P(\zeta)=c_j\zeta^j+c_{j+1}\zeta^{j+1}+\cdots+c_s\zeta^s, \quad   0<j<k, \quad j\le s, \quad c_j=1, \quad c=c_{\frac{sk}{k-1}}\in \bb C$$
 with $c=0$ whenever $\frac{sk}{k-1}\not\in\bb Z$ or $\lambda\neq 1$ and $gcd\{k,p\mid  c_p\neq 0\}=1$.\\
  Moreover (Remark 4.3 \cite{OT08}) intermediate surfaces of index one are exactly those for which $(k-1)\mid s$. Therefore $S$ of index $m\ge 2$ is associated to a germ
$$F(z,\zeta)=(\lambda \zeta^s z+\sum_{p=j}^s c_p\zeta^p,\zeta^k)$$
with $(k-1)\nmid s$. For a positive divisor $q$ of $k-1$ such that $k-1$ divides $qs$, the authors define a new contracting germ in the following way: set
$$r':= \lfloor \frac{qj}{k}\rfloor,\quad s':=qs-r'(k-1),\quad j'=qj-r'k, \quad P'(\zeta):=\sum_{p=j'}^{s'} c_p \zeta^{qp-r'k}$$
then
$$F'(z,\zeta):=(\lambda\zeta^{s'}+P'(\zeta),\zeta^k)$$
is associated to a surface of index one with action of $\bb Z/q\bb Z$ given by $(z,\zeta)\mapsto (\epsilon^{-r}z, \epsilon\zeta)$, where $\epsilon$ is a primitive root of unity of order $q$.\\
As the index of  an intermediate surface is given by \cite[remark 4.5]{OT08}
$$m=\frac{k-1}{gcd\{k-1,s\}}$$
we may take $q=m$ and obtain the expression of a germ associated to $S'$. We notice that  the  second members of  contracting germe $F$ and $F'$ are equal, therefore $k(S)=k(S')$.
\end{proof} 
\section{Surfaces in class VII$_0^+$ with twisted logarithmic 1-form}
\begin{example} Let $S$ be a Inoue-Hirzebruch surface with maximal divisor $D=\G_++\G_-$ composed of two cycles of rational curves. The surface $S$ is defined by a contracting germ $F(z_1,z_2)=(z_1^pz_2^q,z_1^rz_2^s)$ (see\cite{DO} for details). The matrix $\left(\begin{array}{cc}p&q\\r&s\end{array}\right)$ has two eigenvalues $\l_1$ and $\l_2$ which satisfy $\l_1>1$ and $0<|\l_2|<1$ with eigenvectors $x_1=\left(\begin{array}{c}a_1\\b_1\end{array}\right)$ and  $x_2=\left(\begin{array}{c}a_2\\b_2\end{array}\right)$ respectively such that, $a_1$, $b_1$ are of the same sign and $a_2$, $b_2$ of opposite sign, say $a_1>0$, $b_1>0$, $a_2>0$, $b_2<0$. We have then for each curve $C$ in the universal covering space $\tilde S=\hat S$ two  functions
$$G_{C,i}(z)=a_i\log|z_1|+b_i\log|z_2|, \quad i=1,2.$$
Recall that the inverse image of $D$ in $\tilde S$ gives two infinite chains of rational curves $\tilde D=\tilde \G_++\tilde\G_-$.
The function $G_{C,1}$ yields a PSH function $\hat u_1$ on $\tilde S$ with polar set 
 $\tilde D$, PH outside $\tilde D$; the function $G_{C,2}$ yields a PSH function $\hat u_2$ outside $\tilde D$ with value $-\infty$ on one chain, say $\tilde\G_-$, and $+\infty$ of the other $\tilde\G_+$. In the first situation $\hat u_1<0$ is negative, however $\hat u_2$ takes all values of $[-\infty,+\infty]$, in particular there is a Levi flat real hypersurface $\hat H=\{\hat u_2=0\}$ which splits $\tilde S$ into two parts and gives a decomposition of $\tilde S$ into 
 $$\tilde S=\{\hat u_2=-\infty\}\cup \{\hat u_2<0\}\cup \tilde H\cup \{\hat u_2>0\}\cup \{\hat u_2=+\infty\}.$$
 As all these parts of the decomposition are invariant by the automorphism of the covering we derive a decomposition
 $$S=\G_-\cup S_-\cup H\cup S_+\cup \G_+.$$
 The current $dd^c\hat u_1\ge 0$ on $\tilde S$ and $dd^c\hat u_2\ge 0$ on $\tilde S\setminus \tilde\G_+$.
\end{example}

Suppose there exists a  locally integrable function $\tau$ such that $d_\beta d^c_\beta\tau\ge 0$, $\hat u=e^{-\hat f}\hat \tau$ is a multiplicatively automorphic PSH function $\hat u:\hat S\to \bb R\cup \{-\infty\}$ with $\hat D$ as polar set, PH outside $\hat D$  and multiplicative constant $\mu>1$. We want to prove that $\hat u$ is negative.\\

We denote by $\cal T(S)$ the subset of $H^1_{dR}(S,\bb R)$ of classes of Lee forms of locally conformally symplectic forms taming the complex structure $J$
$$\cal T(S)=\{a\in H^1_{dR}(S,\bb R)\mid \exists \o\in\O^2(S), \o^{1,1}>0, d_\a \o=0\}.$$
By \cite{AD23}, $\cal T(S)$ is represented by an open interval contained in $(-\infty,0)$.
Recall that given $\a\in a$ and $\hat f$ on $\hat S$ such that $\hat\a=d\hat f$ we have $\hat f \circ\g-\hat f=c<0$ for the suitable choice of generator of $Aut_S(\hat S)$. The constant $c$ depends only on the class $a$. Denote by $\f:]-\infty, 0[\to \bb R_+$ the decreasing mapping $a\mapsto C=e^{-c}$ which sends $]-\infty,0]$ onto $[1,+\infty[$.\\

Recall that by lemma 6.11 of \cite{AD23} (see also the remark 5.9 in \cite{lcs1}) $S$ in the class VII$_0$ admits a non-trivial twisted holomorphic 1-form $\t\in H^0(S,\O^1\ot \cal L)$ only if $S$ is a Hopf surface, an Inoue-Bombieri surface, or a Enoki surface. In the following theorem we shall suppose that $D\neq 0$.

\begin{thm}\label{logformPSH} Let $S$ in class VII$_0^+$. The following two conditions are equivalent:
\begin{itemize}
\item There exists $0\neq \theta\in H^0(S,\O^1(\log D)\ot \cal L_\mu)$, $\mu>1$, $D\neq 0$,
\item There exists a negative locally integrable function $\tau$ such that $d_\beta d^c_\beta\tau\ge 0$, and for $\hat \b=d\hat f$, $\hat u=e^{-\hat f}\hat \tau$ is a multiplicatively automorphic negative PSH function $\hat u:\hat S\to \bb R_-\cup\{-\infty\}$ with polar set $\hat D\neq 0$, PH outside $\hat D$  and multiplicative constant $\mu>1$.
\end{itemize}
Moreover if one of these conditions is satisfied, 
\begin{itemize}
\item $S$ is a Kato surface, in particular the torsion $\mu$ is an   integer  if $S$ is intermediate and $\mu>1$ is a quadratic number if $S$ is a Inoue-Hirzebruch surface,
\item The PSH function $\hat u$ on $\hat S$ is unique up to multiplication by a positive constant.
\end{itemize}
\end{thm}
\begin{proof}  1) If there exists a non trivial logarithmic form $\t\in H^0(S,\O^1(\log D)\ot \cal L_\mu)$, with $\mu>1$, $D$ contains (at least) a cycle of rational curves \cite[lemma 6.10]{AD23}. As $\mu>1$, $\mu=\f(\b)$ with $\b<0$. Besides, with notations of \cite{AD23}, we have the relation
$$CS(\cal F, D_i, D_i\cap D_{i-1})=-\frac{Res_\b(\t)_{\hat D_{i-1}}}{Res_\b(\t)_{\hat D_i}}$$
Since Camacho-Sad index are negative by theorem \ref{periodesmemesigne}, quotients of residues are positive and up to multiplication of $\t$ by a suitable complex number, residues are all positive. By  Proposition \ref{lemma6.18AD23}, there exists $\tau$  such that $T:=d_\b d_\b^c\tau\ge 0$. Recall that for $\hat\b=d\hat f$, 
$$dd^c(e^{-\hat f}\hat\tau)=e^{-\hat f}d_{\hat\b}d_{\hat\b}^c\hat\tau$$
therefore $\hat u=e^{-\hat f}\hat \tau$ is PSH multiplicatively automorphic with positive constant $\mu>1$ and PH outside $\hat D$. The set $\cal T(S)$ of LCS structures is an interval which satisfies $\cal T(S)\subset ]-\infty,0[$ with $\b\not\in\cal T(S)$ therefore $\cal T(S)\subset ]-\infty, b[$ or $\cal T(S)\subset ]b,0[$. In the first case $\tau<0$ by \cite[Prop 3.8]{AD23}. In the second case,  by \cite[Prop 3.7]{AD23} there exists for any $\b\in b$,  a weakly positive current  $\tau'>0$ of degree zero such that $T=d_\b d_\b^c \tau'\ge 0$ and $T\neq 0$. Hence $\hat u=e^{-\hat f}\hat \tau'$ is a non constant PSH multiplicatively automorphic function with positive constant $\mu>1$. As in $\hat S$, $\hat D$ contains an infinite chain of rational curves, $\hat u$ is constant on this chain, and $\mu=1$\ldots contradiction and this case does not exist. \\ 

2) Conversely suppose that $\tau$ exists such that $\hat u=e^{-\hat f}\hat\tau$ is negative, PSH, PH outside $\hat D$, with constant $\mu>1$ of automorphy. By Brunella's theorem  $S$ is a hyperbolic Kato surface. 
By lemma 5.4 in \cite{DO},
$$l:={\rm Card} \{\l\in \bb C^\star\mid h^0(S,\O^1(\log D)\ot\cal L_\l)>0\}\leq 2$$
and equality holds \iff $S$ is a Inoue-Hirzebruch surface.
Therefore in all cases there is a unique real, $\l>1$ for which $h^0(S,\O^1(\log D)\ot\cal L_\l)\neq 0$. Moreover any holomorphic foliation is defined by a twisted logarithmic $1$-form $\t$ \cite[thm 5.5]{DO}, therefore if multiplicative parameter $\l$ is supposed $\l>1$ it is unique.\\
We consider first the case of intermediate Kato surfaces.
\begin{itemize}
\item  Suppose that there exists a global twisted vector field 
$$X\in H^0(S,\Theta(-\log D)\ot \cal L_{\l(S)}).$$
 This situation happens \iff $H^0(S,K_S^{-1}\ot \cal L_{\kappa(S)})\neq 0$, where $\lambda(S)=k(S)\kappa(S)$, i.e. the index of $S$ is $1$ (see  \cite{DOT03}). By unicity of the holomorphic foliation, $X$ is in the kernel of a twisted logarithmic $1$-form 
 $$\t\in H^0(S,\O^1(\log D)\ot \cal L_{k}),$$
(unique up to multiplication by a scalar)  where $k=k(S)$ is uniquely defined by the intersection matrix $M(S)$ of the rational curves. Recall that irreducible curves of $D$ minus the intersection points are leaves of the foliation. \\
The PSH function $\hat u$ defines a holomorphic $(1,0)$-form  $\hat\o=\part \hat u$ outside $\hat D$, hence a twisted holomorphic $(1,0)$-form outside $D$,  $\omega\in H^0(S\setminus D,\O^1\ot \cal L_\mu)$, with the same twisting factor $\mu$. Therefore
$f=\langle X,\o\rangle$ is a holomorphic section of  $\cal L_{\l(S)\mu}$ on $S\setminus D$. As the intersection matrix $M(S)$ is negative definite, $D$ has a strictly pseudoconvex neighbourhood $U$ and $f$ extends to $S$ by Ma.Kato theorem \cite[Prop 1]{K76}.\\
Suppose that $f\neq 0$.  If the divisor $\{f=0\}$ is not empty, it would be topologically trivial, but it is only possible on Enoki surfaces and we have excluded this case. Hence $f$ does not vanish at all.\\
Consider now the leaves of the foliation on $S\setminus D$. By Remmert-Stein theorem the closure $\bar L$ of a leaf $L$ is analytic or contains at least an irreducible component $\delta$ of $D$. If one leaf accumulates, $f$ vanishes at some point of $\delta$ which is impossible. Therefore  leaves extend through $D$ and yield a different foliation, which is impossible \cite[Thm 5.5]{DO}. Therefore $f\equiv 0$, $\o$ and $\t$ define the same foliation. We can consider $\o/\t$ which is a holomorphic section of $\cal L_{\mu /k}$ on $S\setminus D$. It extends to $S$ by the previous Kato theorem and cannot vanish because $S$ is not Enoki once again. Hence  $\cal L_{\mu /k}$ is holomorphically trivial, $\mu=k$, $\o/\t$ is a constant, and $\o\in H^0(S,\O^1(\log D)\ot \cal L_k)$.\\

\item If the index of $S$ is $m\ge 2$, i.e. $m$ is the least integer for which there exists $\kappa\in \bb C^\star$ such that $H^0(S,K_S^{-m}\ot \cal L_\kappa)\neq 0$,  there is a commutative diagram given by proposition \ref{RevRam}. The left-hand side of this diagram can be completed by fiber products

\DIAGV{60}
{}\n{}\n{\hat S'}\n {}\n{\Earv{p'}{90}}\n{}\n{S'}\nn
{}\n{\Near{\hat c}}\n{}\n {}\n{}\n{\Near{c}}\nn
{\hat S\times_S T}\n{}\n{\Earv{p_T}{90}}\n{}\n{T }\nn
{\Sar{\hat\rho}}\n{}\n{}\n{}\n{\Sar{\rho}}\nn
{\hat S\times_S Z}\n{}\n{\Earv{p_Z}{90}}\n{}\n{Z }\nn
{}\n{\Sear{\hat\pi}}\n{}\n{}\n{}\n{\seaR{\pi}}\nn
{}\n{}\n{\hat S}\n{}\n{\Earv{p}{90}}\n{}\n{S}
\diag

All the mappings of this diagram exist by construction of the fiber products, except $\hat c$. The covering $(\hat S\times_S T,p_T,T)$ is cyclic, therefore $\hat c$ contracts the exceptional curves of the first kind.
Let $\hat u$ be a negative PSH function on $\hat S$, PH outside $\hat D$ with polar set $\hat D$ and constant of automorphy $\mu$. The ramified covering $(Z,\pi,S)$ has possible singularities only over singular points of $D$.
The PSH function $\hat u_T:=\hat u\circ\hat\pi \circ \hat\rho$ is PH outside $\hat D_T:=(\hat\pi \circ \hat\rho)^{-1}(\hat D)$ with polar set $\hat D_T$ which is the maximal divisor in $\hat  S\times_S T$. The cyclic covering $(\hat S\times_S T,p_T,T)$ is smooth however perhaps non minimal. An exceptional curve of the first kind belongs to $\hat D_T$ therefore is an irreducible component of the polar set. The PSH function $\hat u':=\hat u_T \circ (\hat c)^{-1}$ on the complement of the points images of these exceptional curves extends continuously to $\hat S'$ as a PSH function.\\
Let $\hat g:\hat S\to \hat S$ the deck automorphism of $\hat S$ such that $S\simeq \hat S/\{\hat g^p\mid p\in\Z\}$ and $\hat g^\star \hat u=\mu \hat u$.\\
We obtain the lifts $\hat g_Z$ and $\hat g_T$ of $\hat g$ applying universal property of the fiber product.\\

\DIAGV{60}
{}\n{}\n{\hat S'}\n{}\n{\Earv{\hat g'}{90}}\n{}\n{\hat S'}\n{}\n{\Earv{p}{90}}\n{}\n{S'}\nn
{}\n{\Near{\hat c}}\n{}\n{}\n{}\n{\Near{\hat c}}\n{}\n{}\n{}\n{\Near{c}}\nn
{\hat S\times_S T}\n{}\n{\Earv{\hat g_T}{90}}\n{}\n{\hat S\times_S T}\n{}\n{\Earv{p_T}{90}}\n{}\n{T }\nn
{\Sar{\hat \rho}}\n{}\n{}\n{}\n{\Sar{\hat\rho}}\n{}\n{}\n{}\n{\Sar{\rho}}\nn
{\hat S\times_S Z}\n{}\n{\Earv{\hat g_Z}{90}}\n{}\n{\hat S\times_S Z}\n{}\n{\Earv{p_Z}{90}}\n{}\n{Z}\nn
{}\n{\Sear{\hat\pi}}\n{}\n{}\n{}\n{\Sear{\hat \pi}}\n{}\n{}\n{}\n{\Sear{\pi}}\nn
{}\n{}\n{\hat S}\n{}\n{\Earv{\hat g}{90}}\n{}\n{\hat S}\n{}\n{\Earv{p}{90}}\n{}\n{S}
\diag

For $\hat u'=\hat u\circ \hat\pi \circ \hat \rho\circ \hat c^{-1}$, we have $\hat u'\circ \hat g'=\mu \hat u'$, hence  $\hat u'$ has the same torsion constant.

By index one case and thm \ref{kOT} of Oeljeklaus-Toma, $\mu=k(S')=k(S)$.

By the case $m=1$, $\hat\t'=\part \hat u'$ induces on $S'$ a twisted logarithmic $1$-form $\t'\in H^0(S',\O^1(\log D')\ot \cal L_{k'})$ invariant by the action of $\bb Z/m\bb Z$, since it is the case for $\hat u_T$ and $\hat u'$, with polar set the maximal divisor $D'$ of $S'$. Besides, by proposition \ref{kOT}, $\mu=k'=k(S')=k(S)$.\\
Let $\g':\hat S'\to \hat S'$ be the generator of $Aut_{S'}(\hat S')\simeq \bb Z$ such that $\g^\star \hat\t'=k(S)\hat\t'$. The quotient $Z'\simeq S'/\bb U_m$ is a normal surface with has at most isolated quotient singularities image by $\pi'$ of intersection points of the rational curves of $S'$, therefore $\hat\t'$ induces on the regular part $Z'_{reg}$ of $Z'$ a twisted logarithmic $1$-form $\t'$.  The singular points of $Z'$  are Jung-Hirzebruch singularities, hence their desingularization is a chain of rational curves (see \cite{bpv}). Let $P$ such a point, $U$ a Stein neighbourhood of $P$ and $\rho':V\to U$ a minimal desingularization, $E=f^{-1}(P)$. On $V$, simply connected, $\cal L_{k(S)}$ is trivial hence ${\rho'}^\star \t'$ extends on $V$ as a logarithmic $1$-form by \cite[Thm 1.1 and Rem. 1.2]{GKK10} (recall that a reflexive coherent sheaf is normal). Finally $\t'$ extends globally on $T'$ and gives a twisted logarithmic $1$-form $\t\in H^0(S,\O^1(\log D)\ot \cal L_{k(S)})$.
\end{itemize}
 For Inoue-Hirzebruch surfaces the Example 4.5. in \cite{AD23} shows that there exist a PSH function $\hat v$ on $\hat S$ with multiplicative automorphic constant $\l>1$. Moreover there are twisted logarithmic 1-form $\t$ and twisted vector field $X$ with same quadratic real number $\l>1$ by \cite[p1526]{DO}.\\
We show as for intermediate case of index $m=1$ that $\part v$ extends to $S$ with same torsion .

By the previous arguments, if $\hat u$ and $\hat v$ are negative PSH function, PH outside the rational curves, $\hat\t=\part \hat u$ and $\hat\o=\part\hat v$ have the same coefficient of torsion and  induce on $S$ two twisted logarithmic $1$-form $\t$ and $\o$ respectively, in the same space $H^0(S,\O^1(\log D)\ot \cal L_{k(S)})$ of dimension one. Therefore there exists a complex number $t$ such that $\t=t\o$. We have then $\part(\hat u-t\hat v)=0$, i.e. $\hat g:=\hat u-t\hat v$ is antiholomorphic on $\hat S$ and gives an antiholomorphic section $g$ of $\cal L_{k(S)}$. The holomorphic section $\bar g$ of $\cal L_{k(S)}$ vanishes identically because $S$ is not a Enoki surface, therefore $\hat u=t\hat v$ and since both $\hat u$ and $\hat v$ are negative $t$ is a positive constant.
\end{proof}

\begin{defn}  Let $S$ in class VII$_0^+$.
A negative locally integrable function $\hat u:\hat S\to \bb R_-$ multiplicatively automorphic  PSH function  with $\hat D$ as polar set, PH outside $\hat D$ and multiplicative constant $\mu>1$ will be called a Green function.
\end{defn}

Here we give a geometric conjectural meaning of the upper bound $b$ of $\cal T(S)$:
\begin{conj} Let $S$ in class VII$_0$ and $\mathcal T(S)=]a,b[$, $-\infty\le a$, $b<0$ be the moduli space of LCS structures taming the complex structure of $S$. Let $\f$ be the decreasing mapping
$$\f:]-\infty, 0]\to \bb [1,+\infty[, \quad a\mapsto C=e^{-c}$$
  If there is a Green function $\hat u$ on $\hat S$ with multiplicity constant $\l>1$ or if there is a twisted logarithmic $1$-form $\t\in H^0(S,\O^1(\log D)\ot\cal L_\l)$ with twisting constant $\l>1$,  then $\l=\f(b)$.
\end{conj} 

\section{Appendix}
We give here explicit examples for Proposition \ref{kOT}:
\begin{example} 1) For $b_2(S)=3$ there is only one configuration of curves for which the index is different from one, that is $a(S)=(42\  2)=(s_1r_1)=-(D_0^2,D_1^2,D_2^2)$. The opposite intersection matrix is
$$M(S)=\left( \begin{array}{ccc} 2&0&-1\\0&2&-1\\-1&-1&2\end{array}\right).$$
Then $\det M(S)=4$, $k(S)=\sqrt{\det M(S)}+1=3$ and the index is $m(S)=2$ because we have $D_{-K}=\frac{3}{2}D_0+\frac{1}{2}D_1+D_2$. The associated contracting germ is of simple type because there is only $t=1$ tree. By  the procedure which allows to determine the sequence of blow-ups and gluing in the construction of the Kato surface S, given in \cite[p337]{OT08} we have with the notations there, $q=0$, 
$$[\a_1 +1,\underbrace{2,\ldots,2}_{\a_1-2},\underbrace{2,\ldots,2}_{s-j+1}]=[42\ 2]$$
hence $\a_1=3$ and $s=j$.\\
Besides, the index is given by the formula \cite[p331]{OT08}
$$m(S)=\frac{k-1}{gcd\{k-1,s\}}$$
therefore $s$ is odd. Since $c_j=1$ and $j<k=3$ for all normal forms, $j=1$ and the germ associated to $S$ is of the form 
$$F(z,\zeta)=(\lambda z\zeta + \zeta,\zeta^3), \quad \lambda\in\bb C^\star.$$
We determine now the germ associated to $S'$. We follow the computations of \cite[p330]{OT08}. Set
$$r':= \lfloor \frac{mj}{k}\rfloor=0,\quad s':=2s=2,\quad j'=2j=2, \quad P'(\zeta):= c_1 \zeta^{2}=\zeta^2$$
hence
$$F'(z,\zeta)=(\lambda z\zeta^{s'}+P'(\zeta),\zeta^k)=(\lambda z \zeta^2+\zeta^2,\zeta^3)=\bigl(\zeta^2(\lambda z+1),\zeta^3)\bigr)$$
Following procedure given in \cite[337]{OT08} we obtain $F'=\Pi_0\circ \Pi_1\circ \Pi_2\circ\sigma$
with 
$$\s(z,\zeta)=\left((\l z+1)^3-1,\frac{\zeta}{\l z+1}\right),$$
$$ \Pi_2(u,v)=\bigl((u+1)v,v\bigr),\quad \Pi_1(u',v')=(v',u'v'),\quad \Pi_0(v',u')=v',u'v').$$
We have $a(S')=[332]$. We check that $det M(S')=4$ and $D_{-K}=2D_0+D_1+2D_2$ in particular $m(S')=1$.\\

2) For $b_2(S)=4$, there are several configurations of index $m\ge 2$. For instance, suppose that $a(S)=(3\ 3\ 22)=(s_1s_1r_2)=-(D_0^2,D_1^2,D_2^2,D_3^2)$. We have
$$M(S)=\left(\begin{array}{cccc}3&0&-1&-1\\0&3&0&-1\\-1&0&2&-1\\-1&-1&-1&2\end{array}\right)$$
$$\det M(S)=4,\quad k(S)=3, \quad D_{-K}=\frac{5}{2}D_0+\frac{3}{2}D_1+3D_2+\frac{7}{2}D_3, \quad m(S)=2.$$
$$[\a_1+2,\underbrace{2,\ldots,2}_{\a_1-1},\a_2+1,\underbrace{2,\ldots,2}_{\a_2-2},\underbrace{2,\ldots,2}_{s-j+1}]=[33 22]$$
therefore $\a_1=1$, $\a_2=2$, $s-j=1$, and we recover that $b_2(S)=\a_1+\a_2+(s-j)=4$.
Besides $t=1$ therefore the germ associated to $S$ is simple,
$$m(S)=2=\frac{k-1}{gcd\{k-1,s\}}$$
hence $s$ is odd and $j$ is even. We have $j<k=3$ hence $j=2$ and $s=3$.
$$F(z,\zeta)=(\l z\zeta^3 + \zeta^2+c_3\zeta^3,\zeta^3)$$

$$r=\lfloor \frac{2j}{k}\rfloor=1,\quad s'=ms - r(k-1)=4,\quad j'=mj-rk=1,$$
$$ P'(\zeta)=\sum_{p=j}^s c_p\zeta^{mp-rk}=\zeta+c_3\zeta^3$$

$$F'(z,\zeta)=(\l z\zeta^4+\zeta+c_3\zeta^3,\zeta^3)=(\zeta A,\zeta^3),\quad {\rm with}\ A=A(z,\zeta)=\l z\zeta^3+1+c_3\zeta^2$$
Set
$$A^{-3}-1=\zeta^2 B,\quad B=-3c_3+\cdots, \quad {\rm and}\quad BA^{-2}=-3c_3+\zeta C,$$
then a sequance of six blow-ups of the unit ball $B^2\subset \bb C^2$
$$\Pi_0(u',v')=(v',u'v'),\quad \Pi_1(u,v)=(uv,v),\quad \Pi_2(u,v)=(uv,v),\quad \Pi_3(u,v)=(uv+1,v)$$
$$\Pi_4(u,v)=(uv,v),\quad \Pi_5(u,v)=(uv-3c_3,v),\quad \s(z,\zeta)=(CA^{-1},\zeta A)$$
$\s$ is a germ of biholomorphism and
$$F'(z,\zeta)=\Pi_0\circ\Pi_1\circ\Pi_2\circ\Pi_3\circ\Pi_4\circ\Pi_5\circ\s(z,\zeta).$$
Setting $C_i=\Pi^{-1}(p_{i-1})$ with $p_{i-1}\in C_{i-1}$, $i\ge 0$, $p_{-1}$ the origin of the ball and $C_{-1}=\s^{-1}(C_5)$, we obtain the sequence of self intersection (in $S$), $(C_0^2,C_1^2,C_2^2,C_3^2,C_4^2,C_5^2)=(-2,-2,-2,-2,-2,-4)$, i.e. setting by a circular permutation $C_5$ at the first place and denoting $D_0=C_5$, a.s.o. we obtain $a(S)=(42\  2222)=(s_2,r_4)$. We check that $D_{-K}=3D_0+2D_1+4D_2+6D_3+5D_4+4D_5$, therefore the index is $m(S')=1$ as wanted.
\end{example}

\section{Declarations}
\noindent My manuscript has no associated data.\\
{\bf Conflicts of interest:} The author states that there is no conflicts of interests.

\end{document}